\documentclass[11pt]{article}
\usepackage{hyperref}
\usepackage[dvips]{graphics}
\usepackage[dvips]{graphicx}
\usepackage{amssymb,amsmath}
\usepackage{amsmath, amsthm}
\newtheorem{thm}{Theorem}
\newtheorem{lemma}[thm]{Lemma}
\newtheorem{definition}{Definition}[section]
\newtheorem{condition}{Condition}[section]
\newtheorem{assmpn}{Assumption}[section]
\newtheorem{rem}{Remark}[section]
\newtheorem{ex}{Example}[section]
\usepackage{appendix}
\newtheorem{cor}[thm]{Corollary}
\def\ep{\epsilon}

\def\R{{\mathbb{R}}}

\def\xept{X_{\ep,t}}

\def\yept{Y_{\ep,t}}
\def\s{\sigma}

\def\l{\left}
\def\r{\right}
\def\p{\partial}
\def\L{\mathcal L}
\def\ol{\overline}
\def\ul{\underline}
\def\Q{\mathcal Q}
\def\ou{\overline{u}}
\def\uu{\underline{u}}
\def\xep{\bar{x}_\ep}
\def\yep{\bar{y}_\ep}
\def\sep{\bar{s}_\ep}
\def\tep{\bar{t}_\ep}

\begin{document}

\title{Large deviations for multi-scale jump-diffusion processes}

\author{Rohini Kumar\thanks{Work partially supported by National Science Foundation grant DMS 1209363 }\\Department of Mathematics\\Wayne State University \\.\\and\\ .\\Lea Popovic\\Department of Mathematics and Statistics\\ Concordia University}

\date{\today}
\maketitle
\begin{abstract}
\noindent We obtain large deviation results for a two time-scale model of jump-diffusion processes. The processes on the two time scales are fully inter-dependent, the slow process has small perturbative noise and the fast process is ergodic. Our results extend previous large deviation results for diffusions. We provide concrete examples in their applications to finance and biology, with an explicit calculation of the large deviation rate function.
\end{abstract}

\tableofcontents

\medskip 

\noindent \textit{Running head.} Large deviations for multi-scale jump-diffusions.

\section{Introduction}
%\begin{itemize}{\bf
%\item applications are often modeled by jump-diffusion processes
%\item give examples of applications from biology and finance (neuroscience, chemical reactions, volatility models) and what would one want to compute using large deviations
%\item some large deviation techniques for multi-scale processes exists in abstract, but needed conditions are difficult to verify (check \cite{FK06} for references)
%\item this work extends the technique of \cite{FFK12} to processes include jumps, different approach using optimal control methods unclear how to extend to two time-scales?}
%\end{itemize}

For a number of processes in finance and biology the appropriate stochastic modelling is done in terms of multi-scale Markov processes with fully dependent slow and fast fluctuating variables. The most common examples of such multi-scale processes (random evolutions, diffusions, state dependent Markov chains) are all particular cases of jump-diffusions. The law of large numbers limit, central limit theorem, and the corresponding large deviations behaviour of these models are all of interest in applications.

One case of their use in finance is in multi-factor stochastic volatility models, which are used to capture the smiles and skews of implied volatility. The separation of time scales is helpful for calibration, since it allows one to reduce the number of group parameters. The rate function from the large deviation principle for the stock price process can be used to obtain the price of short maturity options, as well as the limit of the at-the-money implied volatility.  These have been explicitly calculated for models in which the logarithm of the stock price and the stochastic volatility are driven by diffusions (\cite{FFK12}, \cite{FFF}). However, much of the empirical evidence (\cite{B-NS01}, \cite{Kou02}) suggests that mean-reverting jump-diffusions would be a more appropriate model for the problem.

In biology one case of their use is in models of intracellular biochemical reactions. Due to low copy numbers of various key molecular types and varying strengths in chemical bonds, normalized copy numbers of different types of molecules  are processes on multiple time-scales (see \cite{BKRP05}, \cite{KK13} for references to the biology literature). Changes in molecular compositions are modelled by state-dependent Markov chains, and on the slower time scale are well approximated by diffusions with small noise or piecewise deterministic Markov chains  (\cite{KKP12}).  The rate function from the large deviation principle for slowly fluctuating molecular species is used to calculate the propensity for switching in a network that has multiple stable equilibria. Since intracellular processes are also subject to other sources of `extrinsic' noise, multiple time-scale diffusions may include jumps from additional sources. For example, there can be errors during cell division (\cite{Pau10}, \cite{Pau11}); a stochastic model combining both reactions and cell division was analyzed in \cite{McSP14}.

%Another example arises in modeling of motoneuron membrane potentials, which are subject to noise from external inputs and changing connectivity of the synaptic network. A reasonable model (\cite{Dit11}, \cite{Sac13}) for the membrane potential are mean-reverting jump-diffusions, while the external inputs and changes in network connectivity are deterministic evolutions dependent on the membrane potentials of all the neurons in the network. The separation of time scales could occur under the assumption that the single neuronal dynamics fluctuates much faster relative to the local field potential that represents the synaptic input to a larger global population or relative to connectivity changes (\cite{Bar10}, \cite{Wain12}). The dynamics of simultaneously recorded neuron firings are then used to provide valuable clues on functional organization in the network. {\bf - delete paragraph if not using an example from neuro at the end}

%\medskip
Large deviation results for multi-scale diffusions have been studied by Freidlin (see \cite{FW98} Chapter 7), Veretennikov (\cite{Ver00}), Dupuis et al. (\cite{Dup12}), and Puhalskii (\cite{Puh15}). For the multi-scale Markov chains where the slow process is a piecewise deterministic Markov processes and the fast process is a Markov chain on a finite state space explicit results were obtained by Faggionato et al. (\cite{Fagg09, Fagg10}). For jump-diffusions there are very few  large deviation results. On a single time scale, there are results  by Imkeller et al. (\cite{Imk09}) for first exit times for SDEs driven by symmetric stable and exponentially light-tail symmetric L\'evy processes. An approach based on control theory and the variational representation was developed by Budhiraja et al. in \cite{Bud11} and extended to infinite dimensional versions \cite{Bud13} (that is, SPDEs rather than SDEs driven by a Poisson random measure). It is not easy to see how to use these results in a multi-scale model of jump-diffusions. A special case of a multi-scale process where the slow process is a diffusion and the fast process is a mean-reverting process driven by a Levy process was studied by Bardi et al. (\cite{BCS}), and the authors use PDE methods to prove asymptotics of an optimal control problem. 

A general method for Markov processes based on non-linear semigroups and viscosity methods was developed by Feng and Kurtz in \cite{FK06}. However, verifying the abstract conditions needed to apply this method to multi-scale jump-diffusions  is a non-trivial task.
In this paper we give a proof of large deviations for two time-scale jump-diffusions, using a technique developed by Feng et al. in \cite{FFK12}. The advantage of this method is that it is constructive and, with some effort, can be tailored to different multi-scale processes. 
Our proof follows the steps of \cite{FFK12}, extending it to processes with jumps and full dependence of the slow and fast components. It is based on viscosity solutions to the Cauchy problem for a sequence of partial integro-differential equations and uses a construction of the sub- and super-solutions to related Cauchy problems as in \cite{FFK12}.  %({\bf anything else we did differently?}). 
Our results hold for slow and fast jump-diffusions which are fully inter-dependent, and where the fast processes is ergodic but not necessarily symmetric. In case the evolution of both processes is spatially homogeneous in the slow variables, we can also provide a more explicit (than a solution to a variational problem) formula for the rate function.
 
\section{Two time-scale jump-diffusion}
%{\bf i added a small drift term in the SDE for $X$}
%{\bf modify to higher $d$? - higher $d$ fine for $X$}
%{\bf if we let $\rho=0$ from the start  $\L_1^x=\L_1^{x,p}$ and $\pi^p$, $Y^{x,p}$, $\mathcal E^{x,p}$  don't need superscript $p$ -- why keep the correlation of diffusive parts if we are keeping the jumps independent?}\\ 
%{\bf why do we write the PRM with jumps depending on $x,y$ rather than the measure?}\\

\noindent Consider the following system of stochastic differential equations:
\begin{subequations}\label{SDEs}
\begin{align}
d\xept=&b(X_{\ep,t-},Y_{\ep,t-})dt+\ep b_0(X_{\ep,t-},Y_{\ep,t-})dt+ \sqrt{\ep}\s(X_{\ep,t-},Y_{\ep,t-})dW_t^{(1)}\nonumber\\
&+\ep\int k(X_{\ep,t-},Y_{\ep,t-},z) \tilde{N}^{\frac1\ep. (1)}(dz,dt),\label{X}\\
d\yept=&\frac{1}{\ep}b_1(X_{\ep,t-},Y_{\ep,t-})dt+\frac{1}{\sqrt{\ep}}\s_1(X_{\ep,t-},Y_{\ep,t-})\l(\rho dW_t^{(1)}+\sqrt{1-\rho^2}dW_t^{(2)}\r)\nonumber\\
&+\int k_1(X_{\ep,t-},Y_{\ep,t-},z)\tilde N^{\frac 1\ep. (2)}(dz,dt),\label{Y}\\
X_{\ep,0}=&\ x_0, \ Y_{\ep,0}=\ y_0,\nonumber
\end{align}
\end{subequations}
 where $N^{\frac1\ep.(1)}(\cdot,\cdot), N^{\frac1\ep.(2)}(\cdot,\cdot)$ are independent Poisson random measures with intensity measures $\nu_1(dz)\times\frac{1}{\ep}dt, \nu_2(dz)\times\frac{1}{\ep}dt$; the L\'evy measures  $\nu_1$ and $\nu_2$ satisfy  $ \int_\R (1\wedge z^2)\nu_2(dz)<\infty$ and  $ \int_\R (1\wedge z^2)\nu_2(dz)<\infty$; the centered versions are defined as $$\tilde N^{\frac1\ep.(1)}(\cdot,\cdot)=N^{\frac1\ep.(1)}(\cdot,\cdot)-\nu_1(dz)\times\frac{1}{\ep}dt, \quad\tilde N^{\frac1\ep.(2)}(\cdot,\cdot)=N^{\frac1\ep.(2)}(\cdot,\cdot)-\nu_2(dz)\times\frac{1}{\ep}dt$$ % from which their centered versions are defined as $$\tilde N^{\frac1\ep.(1)}(\cdot,\cdot)=N^{\frac1\ep.(1)}(\cdot,\cdot)-\nu_1(dz)\times\frac{1}{\ep}dt, \quad \tilde N^{\frac1\ep.(2)}(\cdot,\cdot)=N^{\frac1\ep.(2)}(\cdot,\cdot)-\nu_2(dz)\times\frac{1}{\ep}dt,$$
and  $W^{(1)},W^{(2)}$ are independent Brownian motions independent of $N^{\frac1\ep.(1)}(\cdot,\cdot), N^{\frac1\ep.(2)}(\cdot,\cdot)$.  

 \medskip
 
% We will assume that
%\begin{assmpn}[Jump condition]\label{Levy_measure}
%The L\'evy measure  $\nu_2$ satisfies %finite measures, i.e.  $\nu_1, \nu_2\in\mathcal{M}_f(\R)=\{\mu:\mu(\R)<\infty\}$, 
%$\int_\R (1\wedge z^2)\nu_1(dz)<\infty,
%$ \int_\R (1\wedge z^2)\nu_2(dz)<\infty$ and the L\'evy measure $\nu_1$ is a finite measure. %, and there exist $f(z)>0$ and $k(x,y)>0$  such that 
%\[|k(x,y,z)|\leq k(x,y) f(z) \quad \text{and}\quad \int e^{f(z)}\nu_1(dz)<\infty.\]
%\end{assmpn}
%and have  densities $\nu_1(dz)=\nu_1(z)dz, \nu_2(dz)=\nu_2(z)dz$. %for some probability measure $\nu_1(\cdot)$ on $[0,1]$.
% For simplicity we take for all $x,y$ the increments $k(x,y,\cdot), k_1(x,y,\cdot)$ to be {\bf compactly supported - no need?}, although this could be relaxed in general.
To ensure existence and uniqueness of solutions to the system \eqref{SDEs} we assume
\begin{assmpn}[Lipschitz condition]\label{assmpn:Lip} There exists $K_1>0$ such that $\forall$$(x_1,y_1), (x_2,y_2)\in \mathbb R^2$ 
\begin{equation}\label{Lipschitz}\begin{split}
&|b(x_2,y_2)-b(x_1,y_1)|^2+|b_0(x_2,y_2)-b_0(x_1,y_1)|^2+|b_1(x_2,y_2)-b_1(x_1,y_1)|^2\\
&+|\sigma(x_2,y_2)-\sigma(x_1,y_1)|^2+|\sigma_1(x_2,y_2)-\sigma_1(x_1,y_1)|^2\\
&+\int |k(x_2,y_2,z)-k(x_1,y_1,z)|^2\nu_1(z)dz+\int |k_1(x_2,y_2,z)-k_1(x_1,y_1,z)|^2\nu_2(z)dz\\&\leq K_1(|x_2-x_1|^2+|y_2-y_1|^2).\end{split}\end{equation}
\end{assmpn}

\begin{assmpn}[Growth condition]\label{assmpn:growth}
 There exists $K_2>0$ such that $\forall$$(x,y)\in \mathbb R^2$ 
\begin{equation}\label{growth}\begin{split}
&|b(x,y)|^2+|b_0(x,y)|^2+
|b_1(x,y)|^2+|\sigma(x,y)|^2+
|\sigma_1(x,y)|^2\\&+\int |k_1(x,y,z)|^2\nu_2(z)dz
+\int |k(x,y,z)|^2\nu_1(z)dz
\leq K_2(1+x^2+y^2).
\end{split}
\end{equation}
Define \begin{equation}\label{V}
V(y;x,p):=b(x,y)p+\frac1 2\sigma^2(x,y)p^2 +\int\l(e^{pk(x,y,z)}-1-pk(x,y,z)\r)\nu_1(z)dz.
\end{equation}
For each $x$ and $p$ in $\R$ there exists $K_{x,p}>-\infty$ such that 
\begin{equation}\label{V_bound}%\label{mgfbound}
V(y;x,p)\geq K_{x,p} \quad \forall y\in \mathbb R.
%\int e^{pk(x,y,z)}\nu_1(z)dz<\infty, \; \forall (x,y)\in \mathbb R^2
%\int e^{pk(x,y,z)}\nu_1(z)dz\leq K_{\nu_1,p}(1+x^2+y^2)p^2
\end{equation}
\end{assmpn}

\noindent If existence and uniqueness of solutions to \eqref{X}+\eqref{Y} can be established by other means, we will only assume the growth condition i.e.\ Assumption \ref{assmpn:growth}, that the coefficients are continuous, and the lower bound \eqref{V_bound} on $V$.

%\noindent{\bf what is the right condition on $\nu_1$ so that (\ref{H0}) and (\ref{H1}) are well defined?}\\
%\begin{rem}
%Although we restrict our attention to models satisfying Assumptions \ref{}, the model can easily be extended. For example, we can incude the option pricing model given by Bardorff-Nielsen and Shephard \cite{B-NS2001} where the squared volatility is an Ornstein-Uhlenbeck process driven by a non-Gaussian Levy process.
%\end{rem}
\medskip
\noindent The infinitesimal generator of $(X_\ep,Y_\ep)$ is for $f\in C_b^2(\R\times\R)$ defined by 
\begin{equation}\label{generator}\begin{split}
\L_\ep f(x,y)=&\; b(x,y)\p_xf(x,y)+\rho\s(x,y)\s_1(x,y)\p^2_{xy}f(x,y)\\
&+\ep b_0(x,y)\p_xf(x,y)+\frac{\ep}{2}\s^2(x,y)\p^2_{xx}f(x,y)\\
&+\frac{1}{\ep}\int \l(f(x+\ep k(x,y,z),y)-f(x,y )-\ep k(x,y,z)\p_xf(x,y) \r)\nu_1(z)dz\\
&+\frac{1}{\ep}\Bigl[b_1(x,y)\p_yf(x,y)+\frac{1}{2}\s_1^2(x,y)\p^2_{yy}f(x,y)\\
&+\int \l(f(x,y+k_1(x,y,z))-f(x,y )-k_1(x,y,z)\p_yf(x,y) \r)\nu_2(z)dz\Bigr].
\end{split}
\end{equation}

\noindent Fix $x\in\R$ and let $Y^x$ denote the process satisfying the SDE
\begin{equation}\label{Y^x}\begin{split}
dY_t=&b_1(x,Y_{t-})dt+\s_1(x,Y_{t-})\l(\rho dW_t^{(1)}+\sqrt{1-\rho^2}dW_t^{(2)}\r)\\
&+\int k_1(x,Y_{t-},z)\tilde N^{(2)}(dz,dt), \quad Y^x_0=y_0.\end{split}
\end{equation} This is the SDE \eqref{Y} where $\epsilon$ is set equal to $1$ and $X_{\ep,t}$ is set equal to $x$.
Let $\L_1^x$ denote the generator of $Y^x$, then, for $f\in C_b^2(\mathbb R)$, 
 \begin{equation}\label{L1}\begin{split}\L^x_1f(y):=&b_1(x,y)\p_yf(y)+\frac{1}{2}\s_1^2(x,y)\p^2_{yy}f(y)\\&
+\int\l(f(y+k_1(x,y,z))-f(y )-k_1(x,y,z)\p_yf(x,y) \r)\nu_2(z)dz.\end{split}\end{equation}
For fixed $p\in\R$  define the perturbed $\L^{x,p}_1$ generator for $f\in C_b^2(\R^2)$  by
\begin{equation}\label{pert-gen}\begin{split}
\L^{x,p}_1f(y):&=\l[\rho\s(x,y)\s_1(x,y)p +b_1(x,y)\r]\p_yf(y)+\frac{1}{2}\s_1^2(x,y)\p^2_{yy}f(y)\\&
+\int\l(f(y+k_1(x,y,z))-f(y )-k_1(x,y,z)\p_yf(x,y) \r)\nu_2(z)dz,\end{split}
\end{equation}
and let $Y^{x,p}$ be the process corresponding to the generator $\L_1^{x,p}$. 
For each $x,p\in\R$ we assume the following about $Y^{x,p}$

\begin{assmpn}[Ergodicity condition]\label{assmpn:Ydist}
The process $Y^{x,p}$  is Feller continuous with transition probability $p^{x,p}_t(y_0,dy)$, which at $t=1$ %such that:  {\bf (a)} $p^{x,p}_1(y_0,dy)$ 
has a positive density  $p^{x,p}_1(y_0,y)$ with respect to some reference measure $\alpha(dy)$.%; \\or: 
%{\bf (b)} there is a unique invariant probability measure $\pi^p(x,\cdot)$ with respect to which $p^{x,p}_t(y_0,dy)$ is symmetric (and $\pi^p(x,\cdot)$ is reversible), that is
%\[\int_{y\in\R}\L^{x,p}_1f(y) \pi^p(x,y)dy=0, \forall f\in C^\infty_c(\R).\]
%and
%\[\int f(y)\L_1^{x,p}g(y)\pi^p(x,y)dy=\int g(y)\L_1^{x,p}f(y)\pi^p(x,y)dy, \forall f,g\in C^2(\R).\]
\end{assmpn}

%\noindent {\bf uniqueness of the eigenvalue problem holds if its state space is compact regardless of reversibility of Y, hence we should make an alternative to this assumption on symmetry: state space of $Y$ is compact + Condition B.1, see Theorem B.3 of \cite{FK06}}\\
% In addition, we assume 
\begin{assmpn}[Lyapunov condition]\label{assmpn:Lyap} %Either one of the following conditions hold.
%\begin{enumerate}
 %For each $\theta\in (0,1]$, 
 There exists a positive function $\zeta(\cdot)\in C^2(\R)$, such that $\zeta$ has compact finite level sets, 
and for each compact set $\Gamma\subset \R$, $\theta\in (0,1]$ and $l\in \mathbb R$, there exists a compact set $A_{l,\theta,\Gamma}\subset \mathbb R$ such that 
\begin{equation}\label{Lyapunov}\begin{split}
&\{y\in\R:  -\theta e^{-\zeta}\L_1^{x,p}e^{\zeta}(y)-(|V(y;x,p)|+|b_0(x,y)p|+\sigma^2(x,y))\leq l\}\subset A_{l,\theta,\Gamma}, \quad \forall p\in\Gamma, \forall x\in \R.
\end{split}
\end{equation}
%\end{enumerate}
\end{assmpn}
\begin{rem}
In the case where the domain of $Y$ is compact, we can define $\zeta\equiv 0$ which will satisfy Assumption \ref{assmpn:Lyap}. 
\end{rem}
\begin{rem}
Some arguments are simpler in the special case $Y^{x,p}$ in addition has a unique invariant probability measure $\pi^p(x,\cdot)$ with respect to which $p^{x,p}_t(y_0,y)$ is symmetric and $\pi^p(x,\cdot)$ is reversible, that is
\[\int_{y\in\R}\L^{x,p}_1f(y) \pi^p(x,y)dy=0, \;\forall f\in C^\infty_c(\R).\]
and
\[\int f(y)\L_1^{x,p}g(y)\pi^p(x,y)dy=\int g(y)\L_1^{x,p}f(y)\pi^p(x,y)dy, \;\forall f,g\in C^2(\R)\]

\end{rem}

%\medskip
\subsection{Examples}\label{sec:Examples}
We give some examples of  $Y$ that satisfy Assumption~\ref{assmpn:Ydist} as well as a multiplicative  ergodicity condition of the form \[e^{-\tilde\zeta}\L_1^{x,p}e^{\tilde\zeta}(y)\leq -\tilde\zeta(y) +d\] for $\tilde\zeta$ with compact level sets and some constant $d>0$. 
One needs to know the coefficients of the process $X$ to know whether these examples also satisfy Assumption~\ref{assmpn:Lyap}. 
Define $\tilde{V}^p(x,y):=V(y;x,p)+|b_0(x,y)|+\sigma^2(x,y)$. If $\tilde{V}^p(\cdot, \cdot)$ is a bounded function for bounded $p$, then the multiplicative ergodicity condition is sufficient for Assumption 2.4 to hold.
If $\tilde{V}^p(x,y)$ is an unbounded function but has compact level sets, 
and  if  the  $\tilde{V}$-multiplicative ergodicity condition of the form \[e^{-\zeta}\L_1^{x,p}e^{\zeta}(y)\leq -c\tilde{V}^p(x,y) +d, \;\; \mbox{ for some }c>1, d>0\] is met for $\zeta$ with compact finite level sets, then  it may be possible to use this condition in place of Assumption~\ref{assmpn:Lyap} and obtain all the same results (see Example~\ref{fin-example} and Remark~\ref{rem-lyap}).
%; the last two are used in our examples in Section~\ref{sec:examples}

\begin{ex}
Let $\rho=0$, $b_1(x,y)= - b_1(x)y, \sigma_1(x,y)=\sigma_1(x)$ and $k_1(x,y,z)= \frac{\sigma_1(x)}{\sqrt{b_1(x)}}z-y$%$k_1(x,y,z)= z-y$
, where $b_1(x), \sigma_1(x)>0$ are continuous. %satisfy \eqref{Lipschitz},\eqref{growth}. 
Let $\nu_2(z)=\exp\{-z^2\}$. Since the intensity measure $\nu_1$ is a bounded measure, we use  $N^{(2)}$ instead of the compensated Poisson process %$\nu_2(z)=\exp\{-\frac{b_1(x)z^2}{\sigma_1^2(x)}\}$. 
 $\tilde N^{(2)}$. For each $x\in \mathbb R$, the solution to  
\begin{equation*}
dY^x_t= - b_1(x)Y^x_tdt+\sigma_1(x)dW^{(2)}_t+\int_{\mathbb R- \{0\}} (\frac{\sigma_1(x)}{\sqrt{b_1(x)}}z-Y^x_t)N^{(2)}(dz,dt)
\end{equation*}
has unique invariant probability distribution $\pi(x,dy)=\sqrt{\frac{b_1(x)}{\pi \sigma_1^2(x)}}\exp\{-\frac{b_1(x)y^2}{\sigma_1^2(x)}\}dy$  and $Y^x$ is symmetric with respect to it. Geometric ergodicity  is satisfied by $\tilde\zeta(y):=\frac{b_1(x)}{2\sigma_1^2(x)}y^2$.
\end{ex}

\begin{ex}
Take $\rho=0$ and let $\alpha\in (1,2)$. Let $Z_t$ be a 1-dimensional symmetric Levy  process whose Levy measure is $\nu_2(z)dz=|z|^{-(1+\alpha)}{\bf 1}_{|z|>1}dz$. Its infinitesimal generator is the truncated  fractional Laplacian $-(-\Delta)_{>1}^{\alpha/2}$ defined as 
\[-(-\Delta)_{>1}^{\alpha/2}f(y)=\int_{|z|>1}(f(y+z)-f(y))\frac{1}{|z|^{1+\alpha}}dz, \quad \text{for}\quad  f\in C^2_c(\R).\] Let $\sigma_1(x,y):=a(x)\sigma_1(y)$ where $a(\cdot), \sigma_1(\cdot)>0$ are such that $a(\cdot)$ is continuous and $\sigma_1(\cdot)$ is locally $1/\alpha$-H\"older continuous and $\liminf_{|y|\to\infty}\frac{\sigma_1(y)}{|y|}>0$.  Let
\begin{equation*}
dY^x_t=\sigma_1(x,Y^x_{t-})dZ_t.
\end{equation*}
 Then from Theorem 1.7(i) in \cite{CW14},  $\pi(x,dy):=\frac{\sigma_1(y)^{-\alpha}dy}{\int \sigma_1(y)^{-\alpha}dy}$ is the unique invariant probability measure for the $Y^x$ process and $Y^x$ is $\pi(x,\cdot)$-reversible. From Lemma 3.2 in \cite{CW14}, we get $\tilde\zeta(y):= \ln(1+|y|^\theta)$ for $\theta\in (0,1)$ satisfies the geometric ergodicity condition. The special case of this example with $\sigma_1\equiv 1$ is also considered in \cite{BCS}.
\end{ex}

\begin{ex} Let  $c(z,z')$ be a non-symmetric function such that $0<c_0\le c(z,z')\le c_1$, $c(z,z')=c(z,-z')$ and $|c(z,z'')-c(z',z'')|\le c_2|z-z'|^\beta$ for some $\beta\in(0,1)$. Let $\alpha\in (0,2)$, and $Z_t$ be a $1$-dimensional non-symmetric process  whose %Levy measure is $\nu_2(z)dz=c(y+z,y)|z|^{-(1+\alpha)}dz$. Its 
infinitesimal generator is defined by 
\[\mathcal L^{\alpha}_c f(y) =\lim_{\delta\to 0}\int_{|z|>\delta}(f(y+z)-f(y))\frac{c(y,y+z)}{|z|^{1+\alpha}}dz, \quad \text{for}\quad  f\in C^2_c(\R).\]
%and with $\rho=0$, let
Let 
\begin{equation*}
dY_t=-Y_tdt+dZ_t.
\end{equation*}
Heat kernel estimates from \cite{CZ13} imply this non-symmetric jump diffusion is Feller continuous with a positive transition density $p_t(y_0,y), \forall t>0$ %{\bf check!}.
%\noindent{\bf give an example of $Y$ with jumps that satisfies the positive transition density w/o being reversible ($p_t^x$ shouldn't be symmetric in $y_0,y$, for diffusion this is always true if drift is of gradient-in-y form), see Chen-Kumagai 2010 for symmetric ones +  Chen-Zhang arXiv:1306.5015, 2013 for non-symmetric ones  %SDEs with uniformly elliptic diffusion coeff and symmetric jumps w some properties (incl. symmetric stable processes)}
\end{ex}

\begin{ex} Let $Y^x$ be a birth-death Markov chain with birth rate $r_{+}(y)=\lambda(x)$ and death rate $r_{-}(y)=\mu(x) y$, satisfying $\lambda(x),\mu(x)> 0$. Since its state space is countable its transition density is positive, with a unique reversible invariant distribution $\pi(x,y)=e^{-\lambda(x)/\mu(x)}\frac{(\lambda(x)/\mu(x))^y}{y!}$, $y\in\{0,1,\dots\}$. %{\bf could we let the support of $\alpha(dy)$ depend on x?}
%Let $Y^x$ be the diffusion approximation of its rescaled version, it has drift $b^x_1(y)=\lambda(x)-\mu(x)y$ and diffusion coefficient $\sigma^x_1(y)=\sqrt{\lambda(x)+\mu(x)y}$ and is also close to the model known in the finance literature as that of Cox-Ingersoll-Ross.
\end{ex}

\section{Large deviation principle}\label{sec:large_deviations}
We prove a large deviation principle for $\{X_{\ep,t}\}_{\ep>0}$ as $\ep\to 0$ using the viscosity solution approach to verify convergence of a sequence of exponential generators. 
Define \begin{align}\label{uep}
u^h_\ep(t,x,y):=\ep \ln E\l[e^{\frac{h(X_{\ep,t})}{\ep}}|X_{\ep,0}=x,Y_{\ep,0}=y\r], \end{align}
where $h\in C_b(\R)$, the space of bounded uniformly continuous functions on $\R$.
It can be shown (see \cite{FK06}) that for each $h\in C_b(\mathbb R)$, $u^h_\ep$ is a viscosity solution of the Cauchy problem: 
\begin{equation}\label{PDE}
\begin{split}
\p_tu&=H_\ep u \qquad \text{ in }(0,T]\times\R\times\R,\\
u(0,x,y)&=h(x), \qquad \text{ for }(x,y)\in \R\times\R,
\end{split}
\end{equation}
where the non-linear operator is the exponential generator:
 \begin{equation}\label{Hep}
\begin{split}
&H_\ep u(x,y):=\ep e^{-u/\ep}\L_\ep e^{u/\ep}\\
&= b(x,y)\p_xu(x,y)+\rho\s(x,y)\s_1(x,y)\p^2_{xy}u(x,y)+\frac{1}{2}\s^2(x,y)(\p_{x}u(x,y))^2\\
&+\ep\l[b_0(x,y)\p_xu(x,y)+\frac{1}{2}\s^2(x,y)\p^2_{xx}u(x,y)\r]\\
&+\int\l(e^{\frac{u(x+\ep k(x,y,z),y)-u(x,y)}{\ep}}-1-k(x,y,z)\p_xu(x,y) \r)\nu_1(z)dz\\
&+\frac{1}{\ep}\Bigl[\rho\s(x,y)\s_1(x,y)\p_{x}u(x,y)\p_yu(x,y)
 +b_1(x,y)\p_yu(x,y)+\frac{1}{2}\s_1^2(x,y)\p^2_{yy}u(x,y)\Bigr]\\
&+\int\l(e^{\frac{u(x,y+k_1(x,y,z))-u(x,y)}{\ep}}-1 -\frac{k_1(x,y,z)}{\ep}\p_yu(x,y)\r)\nu_2(z)dz +\frac{1}{2\ep^2}\s_1^2(x,y)(\p_{y}u(x,y))^2.
\end{split}
\end{equation}

%\medskip
%\noindent {\bf if we were to add a coupling of jump terms to SDE for Y, then $H_\ep$ would instead have: $\int\l(e^{\frac{u(x+\ep k(x,y,z),y+k_2(x,y,z))-u(x,y)}{\ep}}-1 -k(x,y,z)\p_xu(x,y)\r)\nu_1(z)dz$, and $\p_tu_0$ would instead have: $\int\l(e^{\p_{x}u_0(t,x)k(x,y,z)+u_1(t,x,y+k_2(x,y,z))-u_1(t,x,y)}-1 -k(x,y,z)\p_xu_0(t,x)\r)\nu_1(z)dz$; and then $V$ would lose the jump term: $\int\l(e^{pk(x,y,z)}-1 -k(x,y,z)p\r)\nu_1(z)dz$, and $\L^{x,p}_1$ would gain an extra jump term: $\int\l(e^{pk(x,y,z)}f(y+k_2(x,y,z)-f(y)-1 -k(x,y,z)p\r)\nu_1(z)dz$}\\

\medskip
\noindent In systems with averaging under the law of large number scaling we can identify the limiting non-linear operator $\ol H_0$ as the solution to an eigenvalue problem for the driving process $Y^x$ obtained from $Y_{\ep}$ with $X_{\ep}=x$ and $\ep=1$.

\medskip
\noindent We first identify $u_0$, the limit of $u_\ep$ as $\ep\to 0$, using heuristic arguments.
Assume \begin{equation}\label{expansion}u_\ep(t,x,y)=u_0(t,x)+\ep u_1(t,x,y)+\ep^2 u_2(t,x,y)+\hdots.\end{equation}  Using the $\ep$ expansion of $u_\ep$, \eqref{expansion}, in equation \eqref{PDE}, and collecting terms of $O(1)$, we get 
\begin{equation}\label{u1}\begin{split}\p_tu_0(t,x)&=b(x,y)\p_xu_0(t,x)+\frac{1}{2}\s^2(x,y)(\p_{x}u_0(t,x))^2\\ &+\int\l(e^{\p_{x}u_0(t,x)k(x,y,z)}-1 -k(x,y,z)\p_xu_0(t,x)\r)\nu_1(z)dz\\
&+\rho\s(x,y)\s_1(x,y)\p_{x}u_0(t,x)\p_yu_1(t,x,y)+b_1(x,y)\p_yu_1(t,x,y)+\frac{1}{2}\s_1^2(x,y)\p^2_{yy}u_1(t,x,y)\\
&+\int\l(e^{u_1(t,x,y+k_1(x,y,z))-u_1(t,x,y)}-1-k_1(x,y,z)\p_yu_1(t,x,y)\r)\nu_2(z)dz\\
&+\frac{1}{2}\s_1^2(x,y)(\p_{y}u_1(t,x,y))^2.
\end{split}
\end{equation}
Please note that as this is merely a formal derivation,  we have ignored some technical details (such as justifying interchanging the limit and integral to get the the second line in the above equation). The rigorous proof that follows shows that this formal derivation is indeed correct.  Denote $\p_xu_0(t,x)$ by $p$ and $\p_tu_0(t,x)$ by $\lambda$. Fix $t,x$ and hence $p$ and $\lambda$. %the above becomes a semi linear integro-differential equation in $y$. 
Using the perturbed $\L_1$ generator (\ref{pert-gen}),
%\begin{equation}\begin{split}
%\L^{x,p}_1f(y)&=\l[\rho\s(x,y)\s_1(x,y)p +b_1(x,y)\r]\p_yf(y)+\frac{1}{2}\s_1^2(x,y)\p^2_{yy}f(y)\\&
%+\int\l(f(y+k_1(x,y,z))-f(y)-k_1(x,y,z)\p_yf(y) \r)\nu_2(z)dz.\end{split}\nonumber
%\end{equation}
%and  
%\begin{equation*}%\label{V}
%V(y;x,p):=b(x,y)p+\frac1 2\sigma^2(x,y)p^2 +\int\l(e^{pk(x,y,z)}-1 -k(x,y,z)p\r)\nu_1(z)dz.
%\end{equation*}
the equation \eqref{u1}  can be written as an eigenvalue problem:
\begin{equation}\label{EVP}
\l(\L_1^{x,p}+V(y;x,p)\r)e^{u_1}=\lambda e^{u_1},
\end{equation}
where $V$ is as defined in \eqref{V}. 
 Note that  the eigenvalue $\lambda$ depends on $x$ and $p$, and that if we write $\ol{H}_0(x,p):=\lambda$ then $u_0$ satisfies
 \[\p_tu_0(t,x)=\ol{H}_0(x,\p_xu_0(t,x)),\]
 In the rigorous proof that follows, we identify the limiting operator $\ol{H}_0$ to be as defined in \eqref{H_0bar} %where we define a single-valued operator $\ol{H}_0(x,p)$ as
 which is shown in \cite{DV75} to be the principal eigenvalue $\lambda$ in \eqref{EVP}.  By the expansion \eqref{expansion}, it is clear that $u_0(0,x)=h(x)$. 

\medskip
\noindent The approach of \cite{FK06} for obtaining the large deviation principle is to prove convergence of nonlinear semigroups associated with the nonlinear operators $H_\ep$. %relies on a comparison principle for the limiting operator $\ol{H}_0$.
 In \cite{FK06} the first step is identifying the limit operator $\ol{H}_0$. Existence and uniqueness of the limiting semigroup is obtained by verifying the `range condition' for the limit operator. This amounts to showing
existence of solutions to the equation $(I-\alpha \ol{H}_0)f=h$ for small enough $\alpha>0$  and sufficiently large
class of functions $h$. Since the range condition is difficult to verify, a viscosity method approach is adopted and the range condition is replaced with a comparison principle condition for $(I-\alpha \ol{H}_0)f=h$.  In the viscosity method, existence of the limiting semigroup is by construction, while uniqueness is obtained via the comparison principle.

The approach in this paper uses convergence of viscosity solutions to the Cauchy problem for PIDEs \eqref{PDE}, and to show existence and uniqueness of the limit one then needs to verify the comparison principle for the Cauchy problem $\p_tu_0(t,x)=\ol{H}_0(x,\p_xu_0(t,x))$, with $u_0(0,x)=h(x)$. 

%The approach of \cite{FK06} using convergence of nonlinear semigroups associated with $H_\ep$ relies on the comparison principle for the same limiting operator $\ol H_0$. In that approach, in addition to convergence, one would also need to show uniqueness of the limiting operator by verifying its `range condition'. This amounts to showing existence of solutions to the equation $(I-\alpha \ol H_0)f=h$ for small enough $\alpha$ and sufficiently large class of functions $h$, by using viscosity solutions and verifying the comparison principle for this equation. The approach here instead uses convergence of solutions to the Cauchy problem for PIDEs \eqref{PDE}, and to show uniqueness of the limit one then needs to verify the comparison principle for the Cauchy problem $\p_tu_0(t,x)=\ol{H}_0(x,\p_xu_0(t,x))$, with $u_0(0,x)=h(x)$. 

\medskip
\noindent In the proof of the comparison principle we will also use a Donsker-Varadhan variational representation (\cite{DV75}) for $\ol H_0$  as follows. 
Let $\mathcal P(\mathbb R)$ denote the space of probability measures on $\R$. Define the rate function $J(\cdot;x,p):\mathcal P(\R)\mapsto \R\cup \{+\infty\}$ by 
\begin{equation}\label{rate}
J(\mu;x,p):=-\inf_{g\in D^{++}(\L_1^{x,p})}\int_\R \frac{\L_1^{x,p}g}{g}d\mu,
\end{equation}
where $D^{++}(\L_1^{x,p})\subset C_b(\R)$ denotes the domain of $\L_1^{x,p}$ with functions that are strictly bounded below by a positive constant. 
Then \cite{DV75} implies that the principal eigenvalue $\ol H_0(x,p)=\lambda$ in \eqref{EVP} is also given by
\begin{equation}\label{H_0bar}
\ol H_0(x,p)=\sup_{\mu\in \mathcal P(\R)}\l(\int V(y;x,p)d\mu(y)-J(\mu;x,p)\r),
\end{equation}
where  $V(y;x,p)=b(x,y)p+\frac{1}{2}\s^2(x,y)p^2 +\int\l(e^{pk(x,y,z)}-1 \r)\nu_1(z)dz$.

\medskip 
\begin{rem}
\noindent In the special case $Y^{x,p}$ also has a reversible invariant measure $\pi^{p}(x,\dot)$, we can use the Dirichlet form representation for $J$. Define the Dirichlet form associated with $Y^{x,p}$ by
%since the process $Y^{x,p}$ corresponding to generator $\L_1^{x,p}$ has a unique invariant probability measure and is symmetric with respect to it, we can define its Dirichlet form by
\[\mathcal E^{x,p}(f,g):=-\int f(y) \L_1^{x,p}g(y) d\pi^p(x,dy).\]
Then, Theorem 7.44 in Stroock \cite{Stroo84} implies that 
\begin{equation}\label{Dirichlet}
J(\mu;x,p)=\begin{cases}
\mathcal E^{x,p}\l(\sqrt{\frac{d\mu}{d\pi^p(x,\cdot)}},\sqrt{\frac{d\mu}{d\pi^p(x,\cdot)}}\r)\quad &\text{if}\quad \mu(\cdot)\ll \pi^p(x,\cdot)\\
+\infty & \text{if}\quad \mu(\cdot)\not\ll \pi^p(x,\cdot).\end{cases}\end{equation}
The variational formula (\ref{H_0bar}) then reduces to the classical Rayleigh-Ritz formula
\begin{equation}\label{H_0barRR}
\ol H_0(x,p)=\sup_{f\in L^2(\pi^p), |f|^2=1}\l(\int V(y;x,p)f^2(y)d\pi^p(x,y)dy+\langle \L_1^{x,p}f,f\rangle\r).
\end{equation}
%In case $Y^{x,p}$ is a Markov jump process, additional conditions do not hold, but an expression for $J$ in terms of the Dirichlet form still exists, see recent results of Dupuis (\cite{Dup15})
\end{rem}

\bigskip
\noindent To sum up, we will prove that:
\begin{lemma}\label{uep-converge} Let $\ol H_0$ be as defined in \eqref{H_0bar}, and suppose the comparison principle holds for the nonlinear Cauchy problem:
\begin{equation}\label{u0-eqn}
\begin{split}
&\p_tu_0(t,x)=\ol H_0(x,\p_xu_0(t,x)), \quad \text{for}\quad (t,x)\in (0,T]\times\R;\\
&u_0(0,x)=h(x).
\end{split}
\end{equation}
 Under the Assumptions~\ref{assmpn:Lip}-\ref{assmpn:Lyap}, the sequence of functions $\{u^h_\ep\}_{\ep>0}$ defined in (\ref{uep}) converges uniformly over compact subsets of $[0,T]\times\R\times\R$  as $\ep\to 0$ to the unique continuous viscosity solution $u^h_0$ of \eqref{u0-eqn}. %where 
%\begin{align*}&\ol{H}_0(x,p):=\\
%&\inf_{0<\theta<1,\xi\in C^2_c(\R)}\sup_{y\in\R}\{b(x,y)p+\frac1 2 \s^2(x,y)p^2+(1-\theta)e^{-\xi}\L_1^{x,p}e^{\xi}(y)+\theta e^{-\zeta}\L_1^{x,p}e^{\zeta}(y)\},\end{align*} and $\zeta$ in the above expression is the +\int\l(e^{pz}-1-pz \r)\nu_1(z)dz function that satisfies Assumption \ref{assmpn:Lyap}.
\end{lemma}

\begin{lemma}\label{exp-tight}
The sequence of processes $\{X_{\ep,t}\}_{\ep>0}$ is exponentially tight.
\end{lemma}

%Let $u_0$ in Lemma \ref{uep-converge} be written as $u_0^h$ to indicate its dependence on the initial condition $h$. 
\begin{thm}\label{LDP}
Let $X_{\ep,0}=x_0$, and suppose all the Assumptions %~\ref{assmpn:Lip}-\ref{assmpn:Lyap} and those 
from Lemma~\ref{uep-converge} hold. Then, $\{X_{\ep,t}\}_{\ep>0}$ satisfies a large deviation principle with speed $1/\ep$ and good rate function 
\begin{equation}\label{rate_fn1}I(x,x_0,t)=\sup_{h\in C_b(\R)}\{h(x)-u^h_0(t,x_0)\}.\end{equation}
\end{thm}
\begin{proof}
By Bryc's theorem (Theorem 4.4.2 in \cite{DeZ98}), Lemmas \ref{uep-converge} and \ref{exp-tight}  give us a large deviation principle for $\{X_{\ep,t}\}_{\ep>0}$ as $\ep\to 0$ with speed $1/\ep$ and good rate function $I$ given by \eqref{rate_fn1}. 
\end{proof}

%\begin{rem} 
One of the key conditions for Lemma  \ref{uep-converge} requires one to check that the comparison principle holds for $\ol H_0$.  This condition cannot be established using only the general Assumptions~~\ref{assmpn:Lip}-\ref{assmpn:Lyap}, and needs to be verified on a case by case basis. However, standard theory of comparison principles for viscosity solutions (Theorem 3.7 and Remark 3.8 in Chapter II of \cite{BD97}) implies that it does hold for \eqref{u0-eqn} as soon as $\bar H_0$ is uniformly continuous in $x,p$ on compact sets (see Lemma~\ref{sec:appendix-comparison} of the Appendix). In some cases $\bar H_0$ can be explicitly calculated (see Example~\ref{bio-example}) and continuity directly verified. In other cases one may need to resort to proving that the expression as on the right-hand side of \eqref{app:barH0} is non-positive, using the specifics for the case at hand. %(see Example~\ref{fin-example} {\bf ??}).
%\end{rem}
\begin{cor}\label{Cor:comp_principle} Any of the following separate sets of conditions are sufficient for the comparison principle for the non-linear Cauchy problem (\ref{u0-eqn}) to hold:
\begin{itemize}
\item[(i)] $\bar H_0$ is uniformly continuous in $x,p$ on compact sets;
\item[(ii)] the coefficients $b_1(x,y),\sigma_1(x,y), k_1(x,y,z)$ are independent of $x$, %$\sigma_1^2(y)\geq \epsilon I$ some $\epsilon>0$, 
the coefficients $b(x,\cdot), \sigma(x,\cdot)$ are bounded (bounded functions of $y$) for each $x$, and $\rho=0$ i.e. the correlation between the Brownian motions driving $X$ and $Y$ is $0$. 
\end{itemize}
\end{cor}

\begin{proof} For {\it (i)} see Lemma~\ref{sec:appendix-comparison} of the Appendix which is based on Theorem 3.7 and Remark 3.8 in Chapter II of \cite{BD97}.\\
 For {\it (ii)}  we can directly verify that under these conditions $\overline{H}_0(x,p)$, given in \eqref{H_0bar},  is uniformly continuous on compact sets of $x$ and $p$. For this, first observe that under the conditions in {\it (ii)} the rate function $J$ in \eqref{H_0bar} will be independent of $x$ and $p$. Additionally, $\int V(y;x,p) d\mu(y)$ is uniformly Lipschitz in $x$ and $p$ (uniform over all $\mu\in \mathcal P(\mathbb R)$ ), over compact sets of $x$ and $p$. Finally, since the supremum of  uniformly Lipschitz functions is uniformly  continuous over compact sets, we have the result. %see Lemma 11.60 in \cite{FK06}.
\end{proof}
Note that in Corollary \ref{Cor:comp_principle}, condition {\it (i)} is a more general condition and {\it (ii)} is a sufficient condition (on the coefficients of the model)  under which condition {\it (i)} holds.\\

In very special cases, we can further simplify the expression for the rate function:
\begin{cor}
 If the coefficients in the SDE \eqref{SDEs} are independent of  $x$, then $\ol{H}_0(x,p)$ becomes $\ol{H}_0(p)$ and by Lemma D.1 in \cite{FFK12}, we get 
\begin{equation}\label{rate-fn-no-x}I(x,;x_0,t)= t\ol{L}_0\l(\frac{x_0-x}{t}\r),\end{equation}
where $\ol{L}_0(\cdot)$ is the Legendre transform of $\ol{H}_0(\cdot)$.\\
 
% \noindent {\bf find conditions under which we can simplify the rate function expression when there is x-dependence using the Lagrangian, see "http://arxiv.org/pdf/1406.3605v1.pdf" }
\end{cor}

%\begin{equation}\label{Hep}
%\begin{split}
%H_\ep& u(x,y):=\ep e^{-u/\ep}\L_\ep e^{u/\ep}\\
%&= b(x,y)\p_xu(x,y)+\rho\s(x,y)\s_1(x,y)\p^2_{xy}u(x,y)+\frac{1}{2}\s^2(x,y)(\p_{x}u(x,y))^2\\
%&+\frac{\ep}{2}\s^2(x,y)\p^2_{xx}u(x,y)\\
%&+\int\l(e^{\frac{u(x+\ep z,y)-u(x,y)}{\ep}}-1 \r)\nu_2(x,y,z)dz+\ep \int\l(e^{\frac{u(x+z,y)-u(x,y)}{\ep}}-1 \r)\nu_1_2(x,y,z)dz\\
%&+\frac{1}{\ep}\Bigl[\rho\s(x,y)\s_1(x,y)\p_{x}u(x,y)\p_yu(x,y)
 %+b_1(x,y)\p_yu(x,y)+\frac{1}{2}\s_1^2(x,y)\p^2_{yy}u(x,y)\Bigr]\\
%&+\int\l(e^{\frac{u(x,y+z)-u(x,y)}{\ep}}-1 \r)\nu_1_3(x,y,z)dz +\frac{1}{2\ep^2}\s_1^2(x,y)(\p_{y}u(x,y))^2.
%\end{split}
%\end{equation}

\bigskip
\noindent The proof of Lemma \ref{uep-converge} takes up the bulk of the paper, and consists of the following steps. 
\begin{enumerate}
\item[(Sec \ref{sec:rig})]
\indent\indent $\bullet$ By taking appropriate limits of solutions $u^h_\ep$ to the Cauchy problem (\ref{PDE}) we construct upper-semicontinuous and lower-semicontinuous functions $\ol u^h$ and $\ul u^h$, respectively;\\
\indent $\bullet$ Using an indexing set $\alpha\in\Lambda$, we construct a family of operators $H_0(\cdot\, ; \alpha)$ and $H_1(\cdot\, ; \alpha)$,  in such a way that the upper-semicontinuous function $\ol u^h$ is a subsolution to the Cauchy problem for the operator $\inf_{\alpha \in \Lambda}\{H_0(\cdot\, ; \alpha)\}$, and the lower-semicontinuous function $\ul u^h$ is a supersolution to the Cauchy problem for the operator $\sup_{\alpha \in \Lambda}\{H_1(\cdot\, ; \alpha)\}$. 
\item[(Sec \ref{sec:comp})]
\indent $\bullet$ We prove a comparison principle between subsolutions of $\inf_{\alpha \in \Lambda}\{H_0(\cdot\, ; \alpha)\}$ and 
supersolutions $\sup_{\alpha \in \Lambda}\{H_1(\cdot\, ; \alpha)\}$ above;\\
\indent $\bullet$ We show that this comparison principle implies convergence of solutions $u^h_\ep$ to the Cauchy problem (\ref{PDE}) for $H_\ep$ to solutions $u^h_0$ to the Cauchy problem (\ref{u0-eqn}) for $\ol H_0$.
\end{enumerate}
\noindent The proof of Lemma \ref{exp-tight} uses the estimates obtained in the proof of Lemma \ref{uep-converge} (Section~\ref{sec:exp-tight}).\\

%\noindent {\bf what are the differences wrt proof in \cite{FFK12} other than PDE is replaced by PIDE?}\\

\subsection{Convergence of viscosity solutions of PIDEs}\label{sec:rig}

%{\bf Proof of Lemma~\ref{uep-converge}:}\\
\noindent In Lemma~\ref{uep-converge} we use notions of viscosity solutions, subsolutions and supersolutions. For the standard meaning of these terms, as well as for the definition of the comparison principle, we refer the reader to Definition 4.1 in \cite{FFK12}. Their extension to partial integro-differential equations  (PIDEs) was obtained already in \cite{Alv96} and can be found in \cite{Bar08}.

The proof of convergence of $u^h_\ep$ to $u^h_0$ follows the same steps as Lemma 4.1 in \cite{FFK12} which carries over directly to viscosity solutions of PIDEs.  Because we will need to verify that the conditions there are met, we restate Lemma 4.1 from \cite{FFK12} for viscosity solutions of PIDEs.

Let $\{H_\ep\}_ {\ep>0}$ denote a family of integro-differential operators defined on the domain of functions $\bar{D}_+\cup \bar{D}_-$ where
%on the domains
\begin{eqnarray*}
\bar{D}_{ +} &:=&  \{  f :  f \in C^2(\R^2), \lim_{r \rightarrow \infty} \inf_{|z|>r}f(z)  =+\infty \} \\
\bar{D}_{ -} &:=&  \{ - f :  f \in C^2(\R^2), \lim_{r \rightarrow \infty} \inf_{|z|>r}f(z)  =+\infty \}.
\end{eqnarray*}
Define domains $D_+, D_-$ analogously replacing $\R^2$ by $\R$. 
Consider a class of compact sets in $\R\times\R$ defined by
\[ \Q := \{ K  \times \tilde{K}  : \mbox{compact } K, \tilde{K} \subset \subset \R \}. \]
%For each $f \in C^2(\R^2 )$, let $\nabla f(x,y) \in \R^2$ denote its gradient, $D^2f(x,y) \in M_{2 \times 2}$ the Hessian matrix evaluated at $(x,y)$.
%, and for functions $k,k_1$ and densities $\nu_1, \nu_2$ let 
%\begin{eqnarray*}
%&\I[x,y,f]:= \int \l(f(x+k(x,z))-f(x)-k(x,y,z)\p_yf(x,y)\r)\nu_1(z)dz,\\
%&\I_1[x,y,f]:= \int \l(f(x, y +k_1(x,y,z))-f(x,y)-k_1(x,y,z)\p_yf(x,y)\r)\nu_2(z)dz.
%\end{eqnarray*} 
%Define a sequence of integro-differential operators
%\[H_\ep f(x,y):=H_\epsilon(x,y,\nabla f(x,y), D^2f(x,y)),\]
 %and define first-order differential operators $H_0f(x):= H_0(x, \partial f_x(x))$, for $f\in D_+$ and $H_1f (x):=.\\
Let $u^h_\ep$ be the viscosity solution of the Cauchy problem $\p_t u=H_\ep u$ for the above operator $H_\ep$, with initial value $h$, and define
\begin{definition}\label{relaxed-half-limits}
\begin{align*}
u^h_{\uparrow}(t,x): = 
\sup \{ \limsup_{\epsilon \rightarrow 0+} u^h_\epsilon(t_\ep, x_\epsilon, y_\epsilon): \exists& (t_\ep, x_\epsilon, y_\epsilon) \in   [0,T]\times K \times \tilde{K},\\&  (t_\ep, x_\epsilon) \rightarrow (t,x),  K \times \tilde{K} \in \Q \},\end{align*}
  \begin{align*}
 u^h_{\downarrow}(t,x):=
 \inf \{ \liminf_{\epsilon \rightarrow 0+} u^h_\epsilon(t_\ep, x_\epsilon, y_\epsilon): \exists& (t_\ep, x_\epsilon, y_\epsilon) \in  [0,T]\times K \times \tilde{K},\\& (t_\ep,  x_\epsilon) \rightarrow (t,x), K \times \tilde{K} \in \Q   \}.
\end{align*}
Define $\overline{u}^h$ to be the upper semicontinuous regularization of $u^h_{\uparrow}$, and $\underline{u}^h$ the lower semicontinuous regularization of $u^h_{\downarrow}$.
\end{definition}
 Finally, define the limiting operators (which will be first-order differential operators) $H_0, H_1$ on domains $D_+$ and $D_-$ respectively,  as follows. Let  $\Lambda$ be some indexing set, and  %suppose we have a family of continuous functions %non-linear operators
\begin{align*}
 H_i(x, p; \alpha)&:\R \times \R \mapsto \R, \quad \alpha \in \Lambda, i=0,1. 
% H_\epsilon(z, p,P)&:\R^2 \times \R^2 \times M_{2 \times 2}\mapsto \R.  
  \end{align*}
Define $H_0f(x):= H_0(x, \partial_x f(x))$, for $f\in D_+$ and $H_1f(x):=H_1(x, \p_x f(x)$, for $f\in D_-$, where %for the family $\{H_i(\cdot\,; \alpha)\}_{\alpha\in\Lambda, i=0,1}$
\begin{eqnarray*}
 &H_0(x,p):=\displaystyle\inf_{\alpha\in\Lambda} H_0(x,p;\alpha), \\& H_1(x,p):=\displaystyle\sup_{\alpha\in\Lambda} H_1(x,p;\alpha).\end{eqnarray*}
 Henceforth, with slight abuse of notation, we will refer to $H_i(\cdot, \cdot)$ as operators.\\ %$H_i$, $i=0,1$.\\

Suppose the following conditions hold:

\begin{condition}[limsup convergence of operators]\label{limsupH}
For each $f_0  \in D_+$ and $\alpha \in \Lambda$, there exists $f_{0,\epsilon}   \in \bar{D}_{+}$ (which may depend on $\alpha$)  such that   
\begin{enumerate}
\item \label{cond-a}for each $c>0$, there exists $K \times \tilde{K} \in \mathcal Q$ satisfying  
\begin{equation*}
\{ (x,y) : H_\epsilon f_{0,\epsilon} (x,y) \geq -c \} \cap \{  (x,y) : f_{0,\epsilon}(x,y) \leq c   \}\subset  K \times \tilde{K};
\end{equation*}
\item \label{cond-b}for each $K \times \tilde{K} \in \mathcal Q$,
\begin{equation} \lim_{\epsilon \rightarrow 0} \sup_{(x,y) \in K \times \tilde{K}} |f_{0,\epsilon} (x,y)  - f_0(x) |  =0; \end{equation}
\item \label{cond-c}  whenever $(x_\epsilon, y_\epsilon) \in K \times \tilde{K} \in \mathcal Q$ satisfies $x_\epsilon \rightarrow x$, 
\begin{equation}
 \limsup_{\epsilon \rightarrow 0} H_\epsilon f_{0,\epsilon}(x_\epsilon, y_\epsilon)   \leq H_{0}(x,  \nabla f_0 (x); \alpha).  
 \end{equation}
\end{enumerate}
\end{condition}

\begin{condition}[liminf convergence of operators]\label{liminfH}
For each $f_1  \in D_{-}$ and $\alpha \in \Lambda$, there exists $f_{1,\epsilon} \in \bar{D}_{-}$ (which may depend on $\alpha$) such that 
\begin{enumerate}
\item \label{Kpact} for each $c>0$, there exists $K \times \tilde{K} \in \mathcal Q$ satisfying
 \begin{eqnarray*}
   \{ (x,y) :  H_\epsilon f_{1,\epsilon} (x,y)  \leq c \} 
     \cap \{  (x,y) : f_{1,\epsilon}(x,y) \geq - c   \}\subset   K \times \tilde{K} ;
 \end{eqnarray*}
\item for each $K \times \tilde{K}\in \mathcal Q$, 
\[ \lim_{\epsilon \rightarrow 0} \sup_{(x,y) \in K \times \tilde{K}} |f_1(x) - f_{1,\epsilon} (x,y) | =0;  \]
\item whenever $(x_\epsilon, y_\epsilon) \in  K \times \tilde{K} \in \mathcal Q$, and $x_\epsilon \rightarrow x$, 
\begin{align*}
\liminf_{\epsilon \rightarrow 0} H_\epsilon   f_{1,\epsilon}(x_\epsilon, y_\epsilon) \geq H_{1}(x,  \nabla f_1 (x); \alpha).  \end{align*}
\end{enumerate}
\end{condition}

\noindent In this case the following convergence results for $u^h_\ep$ as $\ep\to 0$ hold.

\begin{lemma}\label{rig-conv}
Suppose the viscosity solutions $u^h_\ep$ to the partial integro-differential equation
\begin{equation}\label{PIDE1}
\p_tu=H_\ep u, \quad u(0,x)=h(x) \nonumber
\end{equation}  
are uniformly bounded, $\sup_{\ep>0}||u^h_\ep||<\infty$. Then, under Condition \ref{limsupH}, $\ou^h$ is a subsolution of 
\begin{equation}\label{sub}
\p_t u(t,x)\leq H_0 (x,\nabla u(t,x))\end{equation}
 and, under Condition \ref{liminfH}, $\uu^h$ is a supersolution of 
 \begin{equation}
\label{super}
\p_t u(t,x)\geq H_1 (x,\nabla u(t,x)).\end{equation}
with the same initial conditions. 
\end{lemma} 
\noindent As the proof is the same as the proof of Lemma 4.1 in \cite{FFK12} we omit it here.
We do need to check Conditions \ref{limsupH} and \ref{liminfH} hold for our problem. This involves identifying the right indexing set $\Lambda$, the family of operators $H_0(\cdot\,;\alpha)$ and $H_1(\cdot\,;\alpha)$, and the appropriate  test functions $f_{0,\ep}$ and $f_{1,\ep}$, for each given $f_0$ and $f_1$, respectively.
%\item Proving a comparison principle between \eqref{sub} and \eqref{super}. 
%For this we will prove %\[\inf_{\alpha\in\Lambda}H_0(\cdot;\alpha)\leq \sup_{\alpha\in\Lambda}H_1(\cdot;\alpha).\]
%\end{enumerate}

\bigskip
\noindent {\bf Verifying Condition \ref{limsupH}:}\label{sec:verify}
As in \cite{FFK12}, we let 
\[\Lambda:=\{(\xi, \theta):\xi\in C^2_c(\R), 0<\theta<1\}\]
 and we define the sequence of operators $H_\ep$ as in (\ref{Hep}) on the domain  
\[D_+:=\{f\in C^2(\R):f(x)=\phi(x)+\gamma \log(1+x^2);\phi\in C^2_c(\R),\gamma>0\}.\]  

\noindent Define the family of operators $H_0(x,p;\xi,\theta)$ for $(\xi,\theta)\in\Lambda$ by
\begin{equation}\label{H0}
\begin{split}H_0(x,p;\xi,\theta):=\sup_{y\in\R}&\{b(x,y)p+\frac1 2 \s^2(x,y)p^2+\int\l(e^{pk(x,y,z)}-1-pk(x,y,z) \r)\nu_1(z)dz\\
&+(1-\theta)e^{-\xi}\L_1^{x,p}e^{\xi}(y)+\theta e^{-\zeta}\L_1^{x,p}e^{\zeta}(y)\}.\end{split}
\end{equation}

\noindent For any $f\in D_+$ and $(\xi,\theta)\in\Lambda$ define a sequence of functions
\[f_{0,\ep}(x,y):=f(x)+\ep g(y),\; \mbox{ where }\; g(y):=(1-\theta)\xi(y)+\theta \zeta(y),\] 
and $\zeta$ is the Lyapunov function on $\R$ satisfying Assumption~\ref{assmpn:Lyap}.
Then, 
\begin{equation}\label{cond1}\begin{split}
H_\ep f_{0,\ep}(x,y)&=b(x,y)\p_xf(x)+\frac1 2 \s^2(x,y)(\p_xf(x))^2+\ep \l(b_0(x,y)\p_xf(x)+\frac 1 2\s^2(x,y)\p^2_{xx}f(x)\r)\\
&\quad+\int\l(e^{\frac{f(x+\ep k(x,y,z),y)-f(x,y)}{\ep}}-1-k(x,y,z)\p_xf(x)\r)\nu_1(z)dz+e^{-g}\L_1^{x,\p_xf(x)}e^{g}(y)\\
&\leq b(x,y)\p_xf(x)+\frac1 2 \s^2(x,y)(\p_xf(x))^2+\ep\l(b_0(x,y)\p_xf(x)+\frac 1 2\s^2(x,y)\p^2_{xx}f(x)\r)\\
&\quad+\int\l(e^{\frac{f(x+\ep k(x,y,z),y)-f(x,y)}{\ep}}-1 -k(x,y,z)\p_xf(x)\r)\nu_1(z)dz\\
&\quad +(1-\theta)e^{-\xi}\L_1^{x,\p_xf(x)}e^{\xi}(y)+\theta e^{-\zeta}\L_1^{x,\p_xf(x)}e^{\zeta}(y).
\end{split}\end{equation}
so, for any sequence $(x_\ep,y_\ep)$ such that $x_\ep\to x$
\[\limsup_{\ep\to 0}H_\ep f_{0,\ep}(x_\ep,y_\ep)\leq H_0(x,\p_xf(x);\xi,\theta),\]  
thus verifying Condition \ref{limsupH}.\ref{cond-c} holds.

\noindent By choice of $D_+$, $f\in D_+$ has compact level sets in $\R$. Also note that $||\p_xf||+||\p^2_{xx}f||<\infty$. Assumption~\ref{assmpn:Lyap} ensures that $- H_\ep f_{0,\ep}(x,\cdot)$ has compact level sets for all $x$ in compact sets.  This proves  Condition \ref{limsupH}.\ref{cond-a} holds.  Condition \ref{limsupH}.\ref{cond-b} is obvious by choice of functions $f_{0,\ep}$.

\bigskip
\noindent {\bf Verifying Condition \ref{liminfH}}:
is exactly the same as verifying Condition \ref{limsupH}, except that the sequence of operators $H_\ep$ are now defined on the domain 
\[D_-:=\{f\in C^2(\R):f(x)=\phi(x)-\gamma \log(1+x^2);\phi\in C^2_c(\R),\gamma>0\};\]  
the family of operators $H_1(x,p;\xi,\theta)$ for $(\xi,\theta)\in\Lambda$ is defined by
 \begin{equation}\label{H1}\begin{split}
H_1(x,p;\xi,\theta):=\inf_{y\in\R}&\{b(x,y)p+\frac1 2 \s^2(x,y)p^2+\int\l(e^{pk(x,y,z)}-1-pk(x,y,z)\r)\nu_1(z)dz\\
&+(1+\theta)e^{-\xi}\L_1^{x,p}e^{\xi}(y)-\theta e^{-\zeta}\L_1^{x,p}e^{\zeta}(y)\};\end{split}
\end{equation}
and for any $f\in D_-$ and $\xi,\theta\in\Lambda$ the sequence $f_{1,\ep}$ is defined as 
\[f_{1,\ep}(x,y):=f(x)+\ep g(y),\; \mbox{ for }g(y):=(1+\theta)\xi(y)-\theta \zeta(y),\] 
so that for any sequence $(x_\ep,y_\ep)$ such that $x_\ep\to x$ we now have
\[\liminf_{\ep\to 0}H_\ep f_{1,\ep}(x_\ep,y_\ep)\geq H_1(x,\p_xf(x);\xi,\theta)\]  
verifies Condition \ref{liminfH}.\ref{cond-c} holds. Conditions \ref{liminfH}.\ref{cond-a} and \ref{liminfH}.\ref{cond-b} hold by the same arguments as above.

\subsection{Comparison Principle}\label{sec:comp}

\noindent The rest of the claim of Lemma~\ref{uep-converge} requires proving uniqueness of solutions to $\p_t u=\ol H_0 u$, with initial value $h$. This can be verified using the comparison principle on the subsolutions and supersolutions of the constructed limiting operators $H_0$ and $H_1$, and the variational representation of $\ol H_0$ from \eqref{H_0bar}. We use the following Lemma 4.2 from \cite{FFK12}.

\begin{lemma}\label{lem-comp}
Let $\ul u^h$ and $\ol u^h$ be defined as in Definition~\ref{relaxed-half-limits}. If a comparison principle between subsolutions of \eqref{sub} and supersolutions of \eqref{super} holds, that is, if every subsolution $v_1$ of \eqref{sub} and every supersolution $v_2$ of \eqref{super} satisfy $v_1\leq v_2$, then $\ul u^h=\ol u^h$ and $u^h_\ep(t,x,y)\to u^h_0(t,x)$, where $u^h_0:=\ul u^h=\ol u^h$, as $\ep\to 0$, uniformly over compact subsets of $[0,T]\times\R\times\R$.
\end{lemma}
\begin{proof}
The comparison principle gives $\ol u^h\leq \ul u^h$, while by construction we  have $\ul u^h\leq \ol u^h$. This gives uniform convergence of $u^h_\ep\to u_0:=\ol u^h=\ul u^h$ over compact subsets of $[0,T]\times\R\times\R$.
\end{proof}

\noindent We next prove the comparison principle for subsolutions of \eqref{sub} and supersolutions of \eqref{super}, that is  every sub solution of 
\[\p_tu(t,x)\leq H_0(x,p):=\inf_{0<\theta<1,\xi\in C^2_c(\R)}H_0(x,p;\xi, \theta),\]
where $H_0$ is as defined in \eqref{H0}, is less than or equal to every super solution of 
\[\p_tu(t,x)\geq  H_1(x,p):=\sup_{0<\theta<1,\xi\in C^2_c(\R)}H_1(x,p;\xi, \theta)\] where $H_1$ is as defined in \eqref{H1}.
%\begin{lemma}\label{comp}
%$\inf_{0<\theta<1,\xi\in C^2_c(\R)}H_0(x,p:\theta,\xi)\leq \sup_{0<\theta<1,\xi\in C^2_c(\R)}H_1(x,p:\theta,\xi)$, i.e.
%\begin{align*}
%\inf_{0<\theta<1,\xi\in C^2_c(\R)}\sup_{y\in\R}\{b(x,y)p+\frac1 2 \s^2(x,y)p^2+(1-\theta)e^{-\xi}\L_1^{x,p}e^{\xi}(y)+\theta e^{-\zeta}\L_1^{x,p}e^{\zeta}(y)\}\\
%\leq \sup_{0<\theta<1,\xi\in C^2_c(\R)}\inf_{y\in\R}\{b(x,y)p+\frac1 2 \s^2(x,y)p^2+(1+\theta)e^{-\xi}\L_1^{x,p}e^{\xi}(y)-\theta e^{-\zeta}\L_1^{x,p}e^{\zeta}(y)\}.
%\end{align*}
%\end{lemma}
%\begin{proof}
We follow the steps in Section 5.2 in \cite{FFK12} with some modifications. The key step is proving %Lemma \ref{comparison-ineq}.  We begin by introducing some notation and  some preliminary definitions.

\medskip
%\begin{lemma}\label{comparison-ineq}
\noindent{\bf Operator Inequality:}\\
\begin{equation}\label{comparison-ineq}
\inf_{0<\theta<1,\xi\in C^2_c(\R)}H_0(x,p;\theta,\xi)\leq \ol H_0(x,p)\leq \sup_{0<\theta<1,\xi\in C^2_c(\R)}H_1(x,p;\theta,\xi),
\end{equation}
 where $\ol H_0(x,p)$ is as defined in \eqref{H_0bar}.
%\end{lemma}
%\begin{proof}

\medskip \noindent Recall the definition of the rate function $J$ from \eqref{rate} and 
variational representation of $\ol H_0$ as
\[\ol H_0(x,p)=\sup_{\mu\in\mathcal{P}(\R)} \l( \int V(y; x,p)d\mu(y)-J(\mu; x,p)\r).\] 
Following steps of Lemma 11.35 of \cite{FK06} (which relies on Assumption~\ref{assmpn:Ydist}) we get that
\[\inf_{0<\theta<1,\xi\in C^2_c(\R)}H_0(x,p:\theta,\xi)\leq \ol H_0(x,p).\]

\noindent From the proof of Lemma B.10 in \cite{FK06}, we have
\[\sup_{0<\theta<1,\xi\in C^2_c(\R)}H_1(x,p:\theta,\xi)\geq \inf_{\mu\in \mathcal P(\R)}\liminf_{t\to\infty} t^{-1} \ln E^\mu\l[e^{\int_0^t V(Y_s^{x,p};x,p)ds}\r].\]
Thus, we need to show that, irrespective of the initial distribution, 
\[\liminf_{t\to\infty} t^{-1} \ln E\l[e^{\int_0^t V(Y_s^{x,p};x,p)ds}\r]\geq \ol H_0(x,p).\]
The proof of this claim depends on the Assumption~\ref{assmpn:Ydist}. We define the occupation measures of the $Y^{x,p}$ process: \[\mu^{x,p}_t(\cdot):=\frac{1}{t}\int_0^t1_{Y_s^{x,p}}(\cdot)ds.\] 
Recall that $\mathcal P(\R)$ is a separable metric space under the Prokhorov metric and that weak convergence of measures is equivalent to convergence in the Prokhorov metric.  Let $\mathcal Q_{t,y_0}$ denote the probability measure on $\mathcal P(\R)$ induced by the occupation measure $\mu_t$ of $Y$ when $Y_0=y_0$. In other words, for $A\in \mathcal B(\mathcal P(\R))$ (the borel sigma-algebra on $\mathcal P(\R)$),
\[\mathcal Q_{t,y_0}(A)=P(\mu_t(\cdot)\in A|Y_0=y_0).\]
\begin{lemma}
 $\inf_{\mu\in \mathcal P(\R)}\liminf_{t\to\infty} t^{-1} \ln E^\mu\l[e^{\int_0^t V(Y_s^{x,p};x,p)ds}\r]\geq \ol H_0(x,p)$.
\end{lemma}
\begin{proof}
Define $\phi:\mathcal P(\R)\to \R$ by $\phi(\mu)=\int V(y;x,p)\mu(dy)$. Take $\tilde{\nu_1}\in\mathcal  P(\R)$,  and let $B(\tilde{\nu_1}, r)$ denote the open ball in $\mathcal P(\R)$ of radius $r$, centered at $\tilde{\nu_1}$. Fix  $\nu_1\in \mathcal P(\R)$, then there exists a compact set $K$ in $\R$ such that $\nu_1(K)>0$. 
The key ingredient in the proof is the uniform LDP lower bound for the occupation measures:
\begin{equation}\label{unif-LDP-LB} \liminf_{t\to\infty}\frac 1 t \log\left[ \inf_{y_0\in K}
\mathcal Q_{t,y_0}(B(\tilde{\nu_1}, r))\right]\geq  -J(\tilde{\nu_1};x,p).\end{equation} This is obtained from Theorem 5.5 in \cite{DV83} under Assumption~\ref{assmpn:Ydist}. While the statement of Theorem 5.5 in \cite{DV83} is in terms of a process level LDP, by the contraction principle it ensures the uniform LDP lower bound  \eqref{unif-LDP-LB}  for the occupation measures $\mu_t^{x,p}$.

We now compute
\begin{align*}
\liminf_{t\to\infty}\frac 1 t \log E^\nu_1\left[e^{\int_0^tV(Y_s^{x,p};x,p)ds}\right]&=\liminf_{t\to\infty}\frac 1 t \log E^\nu_1\left[e^{t\phi(\mu_t^{x,p})}\right]\\
&\geq \liminf_{t\to\infty}\frac 1 t \log E^\nu_1\left[e^{t\phi(\mu_t^{x,p})}1_{\{Y_0\in K\}}\right]\\
&\geq \liminf_{t\to\infty}\frac 1 t \log \left[\inf_{y_0\in K}E^{y_0}\left(e^{t\phi(\mu_t^{x,p})}\right)\right]+\liminf_{t\to\infty}\frac 1 t\log \nu_1(K)\\
&=\liminf_{t\to\infty}\frac 1 t \log  \left[ \inf_{y_0\in K}\int_{\mu\in \mathcal P(\R)}e^{t\phi(\mu)}d\mathcal Q_{t,y_0}(\mu)\right]\\
&\geq \liminf_{t\to\infty}\frac 1 t \log  \left[ \inf_{y_0\in K}\int_{\mu\in B(\tilde{\nu_1}, r)}e^{t\phi(\mu)}d\mathcal Q_{t,y_0}(\mu)\right]\\
&\geq  \inf_{\mu\in B(\tilde{\nu_1}, r)}\phi(\mu)  + \liminf_{t\to\infty}\frac 1 t \log\left[ \inf_{y_0\in K}
\mathcal Q_{t,y_0}(B(\tilde{\nu_1}, r))\right]\\
&\geq \inf_{\mu\in B(\tilde{\nu_1}, r)}\phi(\mu)  -J(\tilde{\nu_1};x,p)
\end{align*}
by \eqref{unif-LDP-LB}. %Theorem 5.5 in [DV83] (for which Assumption~\ref{assmpn:Ydist} was needed).  %(i.e. the lower bound in the LDP for occupation measures $\mu_t^{x,p}$ that holds uniformly  for $x$ in compact sets). 
By Lemma \ref{sec:appendix-phi-lsc} (see Appendix), $\phi$ is a lower semi-continuous function, and so  $\phi(\tilde{\nu_1})\leq \lim_{r\to 0}\inf_{\mu\in B(\tilde{\nu_1}, r)}\phi(\mu)$. Thus taking limit as $r\to 0$  we get

\[\liminf_{t\to\infty}\frac 1 t \log E^\nu_1\left[e^{\int_0^tV(Y_s^{x,p};x,p)ds}\right] \geq \phi(\tilde{\nu_1})-J(\tilde{\nu_1})\] (note that since $V$ is bounded below, $\phi(\mu)>\, - \infty$, and so $\phi(\tilde{\nu_1})-J(\tilde{\nu_1};x,p)$ is well defined and not $-\infty+\infty$). 
Since $\tilde{\nu_1}$ is arbitrary, we get
\[\liminf_{t\to\infty}\frac 1 t \log E^\nu_1\left[e^{\int_0^tV(Y_s^{x,p};x,p)ds}\right] \geq \sup_{\tilde{\nu_1}\in \mathcal P(\R)}\{\phi(\tilde{\nu_1})-J(\tilde{\nu_1};x,p)\}.\] This holds for every $\nu_1\in \mathcal P(\R)$ and so
\[\inf_{\nu_1\in \mathcal P(\R)}\liminf_{t\to\infty}\frac 1 t \log E^\nu_1\left[e^{\int_0^tV(Y_s^{x,p};x,p)ds}\right] \geq \sup_{\tilde{\nu_1}\in \mathcal P(\R)}\{\phi(\tilde{\nu_1})-J(\tilde{\nu_1};x,p)\}.\]
This concludes the proof of the {\bf Operator Inequality} (\ref{comparison-ineq}).
\end{proof}

\begin{rem}
In the special case $Y^{x,p}$ also has a reversible invariant measure $\pi^p(x,\cdot)$ we could also follow the arguments for Lemma 5.4 in \cite{FFK12} using the Dirichlet form representation of $J$ \eqref{Dirichlet}.  

\end{rem}

\begin{proof}[\bf Proof of Lemma \ref{uep-converge}]
By Lemma \ref{rig-conv} and Operator Inequality (\ref{comparison-ineq}), it follows that $\ol u^h$ is a subsolution and $\ul u^h$  a supersolution of the Cauchy problem \eqref{u0-eqn}: $\p_tu(t,x)=\ol H_0(x,\p_x u(t,x))$ with $u(0,x)=h(x)$. If the comparison principle holds for the Cauchy problem \eqref{u0-eqn}, 
%$\bar H_0$ is uniformly continuous on compact sets $\Gamma_\beta\subset \mathbb R^2$, then by Lemma~\ref{sec:appendix-comparison} (see Appendix) the comparison principle holds for $\p_tu(t,x)=\ol H_0(x,\p_x u(t,x))$. 
%uniqueness of $\ol H_0$ and 
then Lemma~\ref{lem-comp} gives us  $\ul u^h=\ol u^h$ and that $u^h_\ep\to u^h_0\equiv \ul u^h=\ol u^h$  uniformly over compact subsets of $[0,T]\times\R\times\R$. 
%\medskip
%\noindent By the definition of $\ol H_0$, and by Assumptions  \ref{assmpn:Lip} and \ref{assmpn:growth}, $\ol H_0(x,p)$ is convex in $p$ and 
%\[|\ol H_0(x,p)-\ol H_0(y,p)|\leq w(|x-y|(1+|p|)),\]
%where $w:[0,\infty)\to [0,\infty)$ is a continuous non-decreasing function with $w(0)=0$ (e.g. $w(u)=K(u+u^2)$ for some large positive function $K$). Then by Theorem 3.7 and Remark 3.8 in Chapter II of \cite{BD97}, the comparison principle holds for the Cauchy problem $\p_tu(t,x)=\ol H_0(x,\p_xu(t,x))$ with initial condition $h$, and so $u_0$ is the unique continuous viscosity solution of \eqref{u0-eqn}. uniformly continuous on compact sets
\end{proof}

\subsection{Exponential tightness}\label{sec:exp-tight}
\begin{proof}[\bf Proof of Lemma \ref{exp-tight}] We prove exponential tightness using the convergence of $H_\ep$ and appealing to supermartingale arguments (see Section 4.5 of \cite{FK06}).

\medskip
\noindent Let $f(x):=\ln(1+x^2)$, so $f(x)\to\infty$ as $|x|\to\infty$, and also $||f^\prime||+||f^{\prime\prime}||<\infty$.  Define $f_\ep(x,y):=f(x)+\ep \zeta(y)$ where $\zeta$ is the positive Lyapunov function satisfying Assumption \ref{assmpn:Lyap} (with $\theta=1$).  Then, for any $c>0$, there exists a compact $K_c\subset \R$ such that $f_\ep(x,y)>c$, $\forall y\in \R$, $\forall x\notin K_c$. %, also for a fixed $(x_0, y_0)$, $f_\ep(x_0, y_0)$ is bounded for all $\ep$.

\medskip
\noindent Observe that by \eqref{cond1} (with $\theta=1$)
\begin{align*}H_\ep f_\ep(x,y)&=\ep e^{-f_\ep/\ep}\L_\ep e^{f_\ep/\ep}\\
&\leq  b(x,y)\p_xf(x)+\frac1 2 \s^2(x,y)(\p_xf(x))^2+\ep\l(b_0(x,y)\p_xf(x)+\frac 1 2\s^2(x,y)\p^2_{xx}f(x)\r)\\
&+\int(e^{\frac{f(x+\ep k(x,y,z),y)-f(x,y)}{\ep}}-1-pk(x,y,z))\nu_1(z)dz + e^{-\zeta}\L_1^{x,\p_xf(x)}e^{\zeta}(y).\end{align*}
By choice of $f$, growth conditions on the coefficients and Assumption~\ref{assmpn:Lyap},  we get there exists $C>0$ such that \[\sup_{x\in\R, y\in\R}H_\ep f_\ep(x,y)\leq C<\infty,\; \forall\ep>0.\]  

\noindent Since\, $e^{\l(f_\ep(X_{\ep,t},Y_{\ep,t})-f_\ep(X_{\ep,0},Y_{\ep,0})\r)/\ep-\int_0^t H_\ep f_\ep(X_{\ep,s},Y_{\ep,s})ds}$ \,is a non-negative local martingale, by optional stopping 
\begin{align*}
P&(X_{\ep,t}\notin K_c)e^{(c-f_\ep(x_0,y_0)-tC)/\ep}\\&\leq E\l[\exp\l\{\frac{f_\ep(X_{\ep,t}, Y_{\ep,t})}{\ep}-\frac{f_\ep(x_0,y_0)}{\ep}-\int_0^tH_\ep f_\ep(X_{\ep,s}, Y_{\ep,s})ds  \r\}\r]\leq 1.
\end{align*}
Therefore for each $c>0$
\[\ep\ln P(X_{\ep,t}\notin K_c)\leq tC-f_\ep(x_0,y_0)-c\]
As $C$ is fixed and independent of $c$ (which we can choose),  $\{X_{\ep,t}\}_{\ep>0}$ is exponentially tight. %satisfies the exponential compact containment condition.

\medskip
%\noindent %We can use the same boundedness of $||f^\prime||, ||f^{\prime\prime}||$ and Assumption~\ref{assmpn:Lyap}, to show that $$\sup_{\ep>0}\sup_{x,y \in K\times\tilde K\subset\subset \R}H_\ep f(x,y)<\infty$$ for any $f\in C^2_c(\R)$ and 
\begin{rem}A similar argument can be used to verify the exponential compact containment condition in Corollary 4.17 in \cite{FK06}, which would give us $\{X_{\ep,\cdot}\}_{\ep>0}$ is exponentially tight.\end{rem}

\end{proof}

\section{Examples}\label{sec:examples}
\subsection{Model for  stock price with stochastic volatility}\label{fin-example}
We consider the stochastic volatility model for stock price suggested by Barndorff-Nielson and Shephard \cite{B-NS01}. Let $X_t$ denote the logarithm of stock price and $Y_t$ the stochastic volatility. 

\begin{equation*}
\begin{split}
dX_t&=(r-\frac1 2Y_t)dt+\sqrt{Y_t}dW_t\\
dY_t&=-\frac{Y_t}{\delta}dt+dZ^{1/\delta}_t,
\end{split}
\end{equation*}
where $W_t$ is a standard Brownian motion and $Z^{1/\delta}_t$ is an independent non-Gaussian L\'evy process with intensity $\frac{1}{\delta}\nu(dz)$; the parameter $0<\delta\ll 1$ denotes the mean-reversion time scale in stochastic volatility. The process $Z$ is often referred to as the {\it background driving L\'evy process }(BDLP). If we are interested in pricing options on the stock which are close to maturity, we will only be interested in small-time asymptotics of the model. We thus scale time by a parameter $0<\epsilon\ll 1$, to get
\begin{equation}\label{SVM}
\begin{split}
dX_{\ep,t}&=\epsilon(r-\frac1 2Y_{\ep,t})dt+\sqrt{\epsilon}\sqrt{Y_{\ep,t}}dW_t\\
dY_{\ep,t}&=-\frac{\ep}{\delta}Y_{\ep,t}dt+dZ^{1/\delta}_{\ep t},
\end{split}
\end{equation}
The multi scale structure comes from the fast mean reversion in stochastic volatility and the small time to maturity. We are interested in the situation where time to maturity ($\epsilon$) is small, but large compared to mean-reversion time ($\delta$) of stochastic volatility. The interesting regime as seen in \cite{FFK12} is when $\delta=\ep^2$. The generator of $(X_\ep, Y_\ep)$ is given by:
\begin{align*}
\L_\ep f(x,y)=&\ep\left((r-\frac1 2y)\p_xf(x,y)+\frac1 2 y \p^2_{xx}f(x,y)\right)\\
&+\frac{1}{\ep}\left( -y\p_yf(x,y)+\int \left(f(x,y+z)-f(x,y)\right)\nu(dz)\right),
\end{align*}
for $f\in C^2_b(\R^2)$.

For this example, since the coefficients are $x$-independent, the perturbed operator $\L_1^{x,p}$ is the same as $\L_1$, the generator of $Y$:  
\[\L_1 f(y)=-y f^\prime(y)+\int \left(f(y+z)-f(y)\right)\nu(dz), \quad \text{ for }f\in C^2_b(\R).\] We can obtain the limiting Hamiltonian $\ol{H}_0$ by solving the eigenvalue problem \eqref{EVP}. Here $V(y;x,p)\equiv V(y;p)=\frac1 2 y p^2$. $\ol{H}_0(p)$ is the eigenvalue $\lambda$ of the eigenvalue problem
\[-y f^\prime(y)+\int \left(f(y+z)-f(y)\right)\nu(dz)+\frac1 2 y p^2 f(y)=\lambda f(y).\]
Note that $f(y)=e^{\frac{p^2}{2}y}$ and $\lambda(p)=\int \left(e^{\frac{p^2}{2}z}-1\right) \nu(dz)$ satisfy the eigenvalue problem.  So $\ol{H}_0(p)=\lambda(p)=\int \left(e^{\frac{p^2}{2}z}-1\right) \nu(dz)$. 
In this example, in the absence of a Lyapunov function $\zeta$ satisfying Assumption \ref{assmpn:Lyap}, we give a slightly altered proof as follows. The following proof assumes $\overline{H}_0(p)$ is finite.

To verify Condition \ref{limsupH}, for $f\in D_+$ and $0<\theta<1$, we define $f_{0,\ep}:= f(x)+\ep ((1-\theta)g(x,y)+\theta \tilde\zeta(y))$, where $g(x,y):=\frac{1}{2}(f^\prime(x))^2 y$ (logarithm of the eigenfunction) and $\tilde\zeta(y):= C^2y$, $C:=\sup_x |f^\prime(x)|$. Then we get
\begin{align*}
H_\ep f_{0,\ep}(x,y)&\leq \frac{1}{2}y(f^\prime(x))^2+\ep((r-\frac{y}{2})f^\prime(x)+\frac{1}{2}yf^{\prime\prime}(x))+(1-\theta)e^{-g}\L_1^{f^\prime(x)}e^g+\theta e^{-\tilde\zeta}\L_1^{f^\prime(x)}e^{\tilde\zeta}\\
&=(1-\theta) \lambda(f^\prime(x))+\theta\l[-y(C^2-\frac{1}{2}(f^\prime(x))^2)+\int (e^{C^2z}-1)\nu(dz)\r]\\
&\qquad+\ep((r-\frac{y}{2})f^\prime(x)+\frac{1}{2}yf^{\prime\prime}(x)).
\end{align*}
Thus $H_\ep f_{0,\ep}$ satisfies Condition \ref{limsupH}.\ref{cond-a}. Condition \ref{limsupH}.\ref{cond-b} is immediate and 
\begin{align*}
\limsup_{\ep\to 0}H_\ep f_{0,\ep}&\leq \inf_{0<\theta<1}\l[(1-\theta) \lambda(f^\prime(x))+\theta \sup_{y}\l(-y(C^2-\frac{1}{2}(f^\prime(x))^2)+\int (e^{C^2z}-1)\nu(dz)\r)\r]\\
&\leq \limsup_{\theta\to 0}\l[(1-\theta) \lambda(f^\prime(x))+\theta \sup_{y}\l(-y(C^2-\frac{1}{2}(f^\prime(x))^2)+\int (e^{C^2z}-1)\nu(dz)\r)\r]\\
&=\lambda(f^\prime(x))=:\overline{H}_0(f^\prime(x)).
\end{align*}

Similarly, to verify Condition \ref{liminfH}, define $f_{1,\ep}:= f(x)+\ep ((1+\theta)g(x,y)-\theta \tilde\zeta(y))$.
It is unnecessary to verify any operator inequality as the limiting operators  $H_0$ and $H_1$ coincide and equal $\overline{H}_0$.

\begin{rem}\label{rem-lyap} Recall the definition of $\tilde{V}$ at the beginning of section \ref{sec:Examples}, $\tilde{V}^p(x,y):=V(y;x,p)+|b_0(x,y)|+\sigma^2(x,y)$.  In general, in case we have a solution to the eigenvalue problem defining the Hamiltonian $\overline H_0$, then the exact same proof as above using $f_{0,\ep}=f(x)+\ep((1-\theta)g(x,y)+\theta \tilde\zeta(y))$, with $g(x,y)$ the logarithm of the eigenfunction and $\tilde\zeta$ satisfying  the $\tilde{V}$-multiplicative ergodicity condition \[ e^{-\tilde\zeta}\L_1^{x,p}e^{\tilde\zeta}(y)\leq -cV^p(x,y)+d, \; \mbox{ for }c>1, d>0\] is enough to conclude our large deviation results  (provided $\tilde{V}$ has compact finite level sets, as it was above). 
\end{rem}

In Barndorff-Nielsen and Shephard \cite{B-NS01}, the BDLP, $Z$, is assumed to have only positive increments. A simple example of such a L\'evy process is a jump process taking finitely many jumps that is the L\'evy measure is 
$\nu(z_i)>0$  where $z_i>0$, $i=1,2,\hdots, k$.   We can then explicitly compute $\ol{H}_0(p)$ and its Legendre transform $\bar{L}(p)$. As seen in \cite{FFK12} (Lemma D.1 in \cite{FFK12}), since $\ol{H}_0(p)$ is not state dependent, we get the rate function to be $I(x,x_0,t)=t\bar{L}\left(\frac{x_0-x}{t}\right)$.
In finance, a common example is where $Z$ is a gamma process, in which case $\nu(dz)=\frac{a}{z}e^{-bz}dz$, $a,b>0$. Then 
\[\ol{H}_0(p)=\begin{cases}
a\ln \l(1+\frac{p^2}{2b-p^2}\r) & \text{ if }-\sqrt{2b}<p<\sqrt{2b}\\
\infty & \text{ if }p^2>2b,
\end{cases}
\] and the rate function is given by $I(x;x_0,t)=t\bar{L}\left(\frac{x_0-x}{t}\right)$, where
\[\bar{L}(q)=\begin{cases}
-a+\sqrt{a^2+2bq^2}-a\ln 2b+a\ln \l(\frac{-2a^2}{q^2}+\frac{2a}{q^2}\sqrt{{a^2}+2b{q^2}}\r) &\text{ if }q>0\\
0 & \text{ if }q=0\\
-a-\sqrt{a^2+2bq^2}-a\ln 2b+a\ln \l(\frac{-2a^2}{q^2}-\frac{2a}{q^2}\sqrt{{a^2}+2b{q^2}}\r) &\text{ if }q<0.
\end{cases}\]
This rate function then gives the asymptotic behavior of a European Call option on the stock. Let $K$ denote the strike price and $S_{\ep,t}=e^{X_{\ep,t}}$, then for $S_0=e^{x_0}<K$ (out-of-the-money call), 
\[\lim_{\ep\to0}\ep\log E\l[S_{\ep,t}-K\r]^+= - I(\log K;x_0,t),\] where maturity time $T=\ep t$.
This follows from Corollary 1.3 in \cite{FFF}.

\medskip
\subsection{Model for self-regulating protein production}\label{bio-example}

 The simplest model for translation of protein from DNA is the system below, with a gene that is  either in its ``on" state $G_1$, or in its ``off" state $G_0$, and in which the protein  activates the changes from ``off" to ``on" state: 
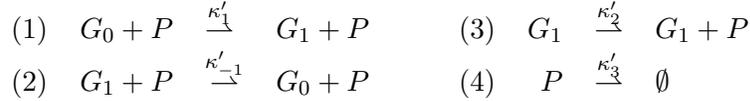
\begin{center}
\begin{tabular}{lrlllrlll} 
(1)&$G_0+P$&$\stackrel {\kappa'_{1}}{\rightharpoonup}$&$G_1+P$&
\qquad(3)&$G_1$&$\stackrel {\kappa'_{2}}{\rightharpoonup}$&$G_1+P$\\
(2)&$G_1+P$&$\stackrel {\kappa'_{-1}}{\rightharpoonup}$&$G_0+P$&
\qquad(4)&$P$&$\stackrel {\kappa'_{3}}{\rightharpoonup}$&$\emptyset$&\\
\end{tabular}
\end{center}

Suppose the amount of protein $P$ is of order $1/\ep$, whose rate of production $\kappa_2'=1/\ep\, \kappa_2$, while its rate of degradation $\kappa_3'=\kappa_3$; where $\kappa_2,\kappa_3$ are of $O(1)$. The amount of genes in the ``on"- and ``off"-state is  $\in\{0,1\}$, their total amount always equaling $1$, and suppose the rates of changes of the gene from the ``on"-state to the ``off"-state and back are very rapid due to its regulation by the large amounts of protein $\kappa_{1}'=\kappa_{1}, \kappa_{-1}'=\kappa_{-1}$, where $\kappa_{1},\kappa_{-1}$ are of $O(1)$.
 This system is characteristic of eukaryotes, where the gene switching noise dominates over the transcriptional and translational noise. 
 We can represent the changes in the system using the process $X_\ep$ for the count of protein molecules normalized by $\ep$, and $Y_\ep$ for the (unnormalized) count of ``on"-gene molecules.
A diffusion process is a good approximation for the evolution of $X_\ep$ as long as the count of proteins is not too small, that is, the unnormalized count is $\gg\ep$ and $X_\ep \sim O(1)$ (\cite{KKP12} gives a rigorous justification of diffusion approximations for Markov chain models that apply in stochastic reaction kinetics). This diffusion solves $dX_{\ep,t}=b(X_{\ep,t},Y_{\ep,t})dt+\sqrt{\ep}\sigma(X_{\ep,t},Y_{\ep,t}) dW_t$ with drift $b(x,y)=\kappa_2 y-\kappa_3 x$ (protein production has only two possible values: it will be 0 when $y=0$, or $\kappa_2$ when $y=1$), with diffusion coefficient $\sigma^2(x,y)=\kappa_2y+\kappa_3 x$, and initial value $X_{\ep,0}=x_0>0$. Changes in the amount of proteins due to other independent sources of noise, such as errors after cell splitting, can be modelled by an additional jump term for $X_\ep$ where the jump measure $\nu_1(dx)$ can be as simple as $\nu_1(z)=\frac 12\delta_{-1}(z)+\frac12\delta_{+1}(z)$, 
producing 
\[dX_{\ep,t}=(\kappa_2 Y_{\ep,t}-\kappa_3 X_{\ep,t})dt +\sqrt{\ep(\kappa_2Y_{\ep,t}+\kappa_3 X_{\ep,t})}dW_t+\ep \int \mathbf 1_{X_\ep>\ep}z\tilde N^{\frac{1}{\ep}}(dz,dt)\]

The amount of genes  $G_1$ in the ``on"-state is a rapidly fluctuating two-state Markov chain $Y$ on $\{0,1\}$ with rates $r_{0\to 1}(x)=\frac 1\ep \kappa_1x$ and $r_{1\to 0}(x)=\frac 1\ep \kappa_{-1}x$ that depend on the normalized amount of protein (note that the amount of genes  $G_0$ in the ``off"-state is $1-Y$). This chain is reversible, and for each $x>0$ it has a unique stationary distribution $\pi^x(1)=1-\pi^x(0)= \kappa_1 /(\kappa_1 +\kappa_{-1})$.

Signalling proteins such as morphogens have to be in the right range of concentrations to avoid triggering the expression of genes at the wrong times. The probabilities of their amounts being out of range are given by the Large Deviation Principle for $X_\ep$ as $\ep\to 0$, for which we need to obtain the solution to the eigenvalue problem for the operator $V(y;x,p)+{\cal L}^x$ where %$V(y;x,p)=b(x,y)p+\frac12 \sigma^2(x,y)p^2$ and 
${\cal L}^x f(y)=r_{0\to 1}(x)\big(f(y+1)-f(y)\big)1_{y=0}+r_{1\to 0}(x)\big(f(y-1)-f(y)\big)1_{y=1}$.\\

In order to solve $(V(y;x,p)+{\cal L}^x)e^{u_1}=\lambda e^{u_1}$ for $\lambda$, let $e^{u_1(x,1)}=a_1(x), e^{u_1(x,0)}=a_0(x)$, for some $a_1,a_0$ strictly positive functions. Then
\begin{eqnarray*}
(\kappa_2-\kappa_3 x)pa_1(x)+(\kappa_2+\kappa_3 x)p^2a_1(x)+\frac12(e^p+e^{-p}-2)a_1(x)+\kappa_{-1}x(a_0(x)-a_1(x))=\lambda a_1(x)\\
-\kappa_3 xpa_0(x)+\kappa_3 xp^2a_0(x)+\frac12(e^p+e^{-p}-2)a_0(x)+\kappa_{1}x(a_1(x)-a_0(x))=\lambda a_0(x)\\
\end{eqnarray*}
equivalently, with $a(x)=a_1(x)/a_0(x)$,
\begin{equation*}
(\kappa_2-\kappa_3 x)p+(\kappa_2+\kappa_3 x)p^2+\kappa_{-1}x(\frac{1}{a(x)}-1)=-\kappa_3 xp+\kappa_3 xp^2+\kappa_{1}x(a(x)-1)
\end{equation*}
which, since $a(x)$ has to be positive, gives
\begin{equation*}
a(x)=\frac{-B+\sqrt{B^2-4AC}}{2A}, \quad A=\kappa_{1}x,\, B=- \kappa_2p-\kappa_2p^2+ (\kappa_{-1}-\kappa_{1})x,\, C=-\kappa_{-1}x
\end{equation*}
and consequently, using notation above,
 \begin{equation*}
\bar H_0(x,p)=-\kappa_3 xp+\kappa_3 xp^2+\kappa_{1}x(a(x)-1)+\frac12(e^p+e^{-p}-2)
 \end{equation*}
Note that when $\kappa_{-1}=\kappa_1$ then 
$$
a(x)=\frac{ \kappa_2p(1+p)+\sqrt{(\kappa_2p(1+p))^2+(2\kappa_{1}x)^2}}{2\kappa_{1}x}
$$
and 
$$
\bar H_0(x,p)=-\kappa_3p(1-p)x+\frac12\kappa_2p(1+p)+\frac12\sqrt{(\kappa_2p(1+p))^2+(2\kappa_{1}x)^2}-\kappa_{1}x+\frac12(e^p+e^{-p}-2).
$$
Note that $\bar H_0(x,p)$ is both convex in $p$ and continuous in $x$.

If one were to use an approximation of the evolution of the normalized protein amount $X_\ep$ by a piecewise deterministic process then (without additional noise) 
\[dX^{\textsc {PDMP}}_{\ep,t}=(\kappa_2 Y_{\ep,t}-\kappa_3 X_{\ep,t})dt\]
while $Y_\ep$ is the same fast Markov chain on $\{0,1\}$. In this case $V(y; x,p)=(\kappa_2-\kappa_3x)p$ and the Hamiltonian (when $\kappa_1=\kappa_{-1}$) becomes
\[\bar H_0^{\textsc {PDMP}}(x,p)=-\kappa_3px+\frac12\kappa_2p+\frac12\sqrt{(\kappa_2p)^2+(2\kappa_{1}x)^2}-\kappa_{1}x.\]
which is easy to compare to the Hamiltonian $\bar H_0$ of the diffusion process $X_\ep$ taking into account the small perturbative noise arising from randomness in the timing of  chemical reactions and from randomness in the outcomes of cell splitting.

\newpage
\appendix

\section{Appendix}

    \renewcommand{\thelemma}{\Alph{section}.\arabic{lemma}}
    \numberwithin{equation}{section}

\begin{lemma}\label{sec:appendix-phi-lsc}
Fix $x,p\in \R$ and let $\phi:\mathcal P(\R)\to \R$ be defined by $\phi(\mu)=\int V(y;x,p)\mu(dy)$. Then, $\phi$ is a lower semi-continuous (l.s.c.) function on $\mathcal P(\R)$. 
\end{lemma}
\begin{proof}
For the rest of the proof, we will write $V(y)$ for $V(y;x,p)$.
Let $V_M:=V\cdot 1_{V\leq M}+M\cdot 1_{V\geq M}$, for $M\geq \inf_yV(y)$. 
To show that $\phi(\mu)$ is l.s.c, it is sufficient to show that if $\mu_n \longrightarrow \mu$ weakly, then $\phi(\mu)\leq \liminf_{n\to\infty}\phi(\mu_n)$. Assume $\mu_n\longrightarrow \mu$ weakly. Then 
\[\int V_M d\mu=\lim_{n\to\infty}\int V_Md\mu_n,\]
by definition of weak convergence of measures, since $V_M$ is a bounded function.  By the monotone convergence theorem we get
\begin{align*}
\phi(\mu)= \int V d\mu&=\lim_{M\to\infty}\int V_M d\mu\\
&=\lim_{M\to\infty}\lim_{n\to\infty}\int V_Md\mu_n\\
&=\sup_M \lim_{n\to\infty}\int V_M d\mu_n\\
&\leq \liminf_{n\to\infty}\sup_M\int V_M d\mu_n\\
&=\liminf_{n\to\infty}\int V d\mu_n  \intertext{by Monotone convergence theorem}
&=\liminf_{n\to\infty} \phi(\mu_n)
\end{align*}

\end{proof}

\begin{lemma}\label{sec:appendix-comparison}
Let $u_1$ be a bounded, upper semicontinuous (u.s.c.), viscosity subsolution and $u_2$ a bounded, lower semicontinuous (l.s.c.), viscosity  supersolution of $\p_tu(t,x)=\ol H_0(x,\p_x u(t,x))$ respectively. If  $u_1(0,\cdot)\le u_2(0,\cdot)$, and $\bar H_0$ is uniformly continuous on compact sets, then $u_1\le u_2$ on $[0,T]\times \mathbb R$ for any $T>0$.
\end{lemma}
\begin{proof} 
Suppose \begin{equation}\label{app:contradiction}
\sup_{t\le T,x}\{u_1(t,x)-u_2(t,x)\}>A\geq \delta>0.\end{equation}
Let $g(t,x)=\ln(1+x^2)+t^2$. 
Define\[\psi(t,x,s,y)=u_1(t,x)-u_2(s,y)-\frac{1}{2} \ln\l(1+\frac{|x-y|^2+|t-s|^2}{\ep}\r)-\beta\l(g(t,x)+g(s,y)\r) -At.\]
Fix $\beta>0$ and let $(\tep,\xep,\sep,\yep)$ denote the point of maximum of $\psi$ in $([0,T]\times\R\times[0,T]\times \R)$ for $\ep>0$. 
Since $u_1, u_2$ are bounded, for fixed $\beta>0$, there exists an $R_\beta>0$ such that $|\xep|,|\yep|\leq R_\beta$ for all $\ep>0$. 

Using \[\psi(\tep,\xep,\tep,\xep)+\psi(\sep,\yep,\sep,\yep)\leq 2 \psi(\tep,\xep,\sep,\yep),\]
we get 
\begin{align*}\frac{1}{2} \ln\l(1+\frac{|\xep-\yep|^2+|\tep-\sep|^2}{\ep}\r)&\leq A(\sep-\tep)+u_1(\tep, \xep)-u_1(\sep,\yep)+u_2(\tep, \xep)-u_2(\sep,\yep)\\
&\leq 2AT+2||u_1||+2|u_2||=:C<\infty,\end{align*} 
which gives us \[|\xep-\yep|^2+|\tep-\sep|^2\leq \ep e^{2C}.\]
Therefore $|\xep-\yep|, |\sep-\tep|\to 0$ as $\ep\to 0$. 

Let \[\phi_1(t,x):=u_2(\sep,\yep)+\frac{1}{2} \ln\l(1+\frac{|x-\yep|^2+|t-\sep|^2}{\ep}\r)+\beta\l(g(t,x)+g(\sep,\yep)\r) +At\] and 
\[\phi_2(s,y):= u_1(\tep,\xep)-\frac{1}{2} \ln\l(1+\frac{|\xep-y|^2+|\tep-s|^2}{\ep}\r)-\beta\l(g(\tep,\xep)+g(s,y)\r) -A\tep.\]
Then $(\tep,\xep)$ is a point of maximum of $u_1(t,x)-\phi_1(t,x)$ and $(\sep,\yep)$ is a point of minimum of $u_2(s,y)-\phi_2(s,y)$. Since $u_1$ and $u_2$ are sub and super solutions respectively, by the definition of sub and super solutions we get
\begin{equation}\label{app:sub}
\frac{\frac{\tep-\sep}{\ep}}{1+\frac{|\xep-\yep|^2+|\tep-\sep|^2}{\ep}}+A+2\beta\tep \leq \ol{H}_0\l(\xep, \frac{\frac{\xep-\yep}{\ep}}{1+\frac{|\xep-\yep|^2+|\tep-\sep|^2}{\ep}}+\frac{2\beta\xep}{1+\xep^2}\r),
\end{equation}
and
\begin{equation}\label{app:super}
\frac{\frac{\tep-\sep}{\ep}}{1+\frac{|\xep-\yep|^2+|\tep-\sep|^2}{\ep}}-2\beta\sep\geq \ol{H}_0\l(\yep, \frac{\frac{\xep-\yep}{\ep}}{1+\frac{|\xep-\yep|^2+|\tep-\sep|^2}{\ep}}-\frac{2\beta \yep}{1+\yep^2}\r).
\end{equation}
Subtracting \eqref{app:super} from \eqref{app:sub}, we get 
\begin{equation}\label{app:barH0}
A+2\beta(\tep+\sep)\leq \ol{H}_0\l(\xep, \frac{\frac{\xep-\yep}{\ep}}{1+\frac{|\xep-\yep|^2+|\tep-\sep|^2}{\ep}}+\frac{2\beta\xep}{1+\xep^2}\r)-\ol{H}_0\l(\yep, \frac{\frac{\xep-\yep}{\ep}}{1+\frac{|\xep-\yep|^2+|\tep-\sep|^2}{\ep}}-\frac{2\beta \yep}{1+\yep^2}\r).
\end{equation}

\medskip
Since $\ol{H}_0(\cdot,\cdot)$ is uniformly continuous over compact sets, and since $|\xep-\yep|\to0$ as \mbox{$\ep\to 0$} (for fixed $\beta$),  the right-hand side of the above inequality goes to $0$ as $\ep\to 0$ and $\beta\to 0$ (note that the terms $\frac{\frac{\xep-\yep}{\ep}}{1+\frac{|\xep-\yep|^2+|\tep-\sep|^2}{\ep}}$,  $\frac{2\xep}{1+\xep^2}$ and $\frac{2\yep}{1+\yep^2}$ are bounded and that $|\xep|,\yep|\leq R_\beta$ for each $\beta$). %This gives
%\[A+2\beta \inf_{\ep}\{\tep+\sep\}\leq 0.\] 

Taking $\ep\to 0$ and then $\beta\to 0$, we get
\[A\leq 0, \] which contradicts \eqref{app:contradiction}. Therefore we must have
\[\sup_{t,x}\{u_1(t,x)-u_2(t,x)\}\leq 0\] which gives us $u_1\leq u_2$.\\

%\noindent By the definition of $\ol H_0$, and by Assumptions  \ref{assmpn:Lip} and \ref{assmpn:growth}, $\ol H_0(x,p)$ is convex in $p$ and 
%\[|\ol H_0(x,p)-\ol H_0(y,p)|\leq w(|x-y|(1+|p|)),\]
%where $w:[0,\infty)\to [0,\infty)$ is a continuous non-decreasing function with $w(0)=0$ (e.g. $w(u)=K(u+u^2)$ for some large positive function $K$). Then by Theorem 3.7 and Remark 3.8 in Chapter II of \cite{BD97}, the comparison principle holds for the Cauchy problem $\p_tu(t,x)=\ol H_0(x,\p_xu(t,x))$ with initial condition $h$, and so $u_0$ is the unique continuous viscosity solution of \eqref{u0-eqn}. 
\end{proof}

%\newpage

\end{document}